\documentclass[12pt,reqno]{amsart}%
\usepackage{amssymb}
\usepackage{amsfonts}
\usepackage{amsmath}
\usepackage{color}
\usepackage{graphicx}%
\usepackage{geometry}
\usepackage{cite}
\usepackage{hyperref}
\usepackage{showkeys}
\usepackage{version}
\providecommand{\U}[1]{\protect\rule{.1in}{.1in}}
\newtheorem{theorem}{Theorem}

\newtheorem{definition}[theorem]{Definition}

\newtheorem{lemma}[theorem]{Lemma}

\newtheorem{proposition}[theorem]{Proposition}
\newtheorem{remark}[theorem]{Remark}

\numberwithin{equation}{section}

\newcommand{\R}{\mathbb R}
\renewcommand{\d}{\mathrm d}
\newcommand{\vep}{\varepsilon}

\newcommand{\LAM}{}
\newcommand{\LAMuno}{1}

\begin{document}

\title[Porous medium equations with non-decreasing constraints]{Porous medium equation with a blow-up nonlinearity and a non-decreasing constraint}
\author{Goro Akagi}
\address{Mathematical Institute, Tohoku University, Aoba, Sendai 980-8578, Japan; Helmholtz Zentrum M\"{u}nchen, Institut f\"{u}r Computational Biology, Ingolst\"{a}dter Landstra\ss e 1, 85764 Neunerberg, Germany; Technische Universit\"{a}t M\"{u}nchen, Zentrum Mathematik, Bolzmannstra\ss e
3, D-85748 Garching bei M\"{u}nchen, Germany.}
\email{akagi@m.tohoku.ac.jp}
\author{Stefano Melchionna}
\address{University of Vienna, Faculty of Mathematics, Oskar-Morgenstern-Platz 1, 1090 Wien, Austria.}
\email{stefano.melchionna@univie.ac.at}

\thanks{\textbf{Acknowledgment.}\quad\textrm{GA is supported by JSPS KAKENHI Grant Number JP16H03946, JP16K05199, JP17H01095 and by the Alexander von Humboldt Foundation and by the Carl Friedrich von Siemens Foundation. SM  acknowledges the support  of the Austrian Science Fund (FWF) project P27052-N25.}}

\maketitle

\begin{abstract}
The final goal of this paper is to prove existence of local (strong) solutions to a (fully nonlinear) porous medium equation with blow-up term and nondecreasing constraint. To this end, the equation, arising in the context of Damage Mechanics, is reformulated as a mixed form of two different types of doubly nonlinear evolution equations. Global (in time) solutions to some approximate problems are constructed by performing a time discretization argument and by taking advantage of energy techniques based on specific structures of the equation. Moreover, a variational comparison principle for (possibly non-unique) approximate solutions is established and it also enables us to obtain a local solution as a limit of approximate ones.
\end{abstract}

\section{Introduction}

In Damage Mechanics, to describe evolution of a \emph{damage variable} $w(x,t)$ for $x \in \Omega \subset \R^d$ and $t > 0$, 
the following \emph{unidirectional gradient flow} of a free energy functional $\mathcal{E}(\cdot)$ defined on, e.g., $L^{2}(\Omega)$ is often used:
\[
\partial_{t}w(x,t)=\Big(-\partial\mathcal{E}(w(\cdot,t))\Big)_{+} \ \mbox{ in } L^2(\Omega), \quad 0<t<T,
\]
where $\partial\mathcal{E}$ stands for a
functional derivative (e.g., subdifferential) of $\mathcal{E}$ and
$$
(\,\cdot\,)_{+}:=\max\{\,\cdot\,,0\}\geq 0
$$
denotes the \emph{positive-part function}. Obviously, any solution $w(x,t)$ to this problem is non-decreasing in time. 
This feature represents unidirectional evolution of damaging phenomena; indeed, the degree of damage is never relaxed
spontaneously. There have already been many contributions to such unidirectional evolutions, starting with the unidirectional heat equation $\partial_{t}w=(\Delta w)_{+}$ (see, e.g.,~\cite{K-S80},~\cite{GiSa94,GiGoSa94} and also recent revisits~\cite{Liu14},~\cite{AkKi13}) and extensions to various nonlinear parabolic equations and systems (see, e.g.,~\cite{BeBi97},~\cite{Fre01,Fr1,Fr2},~\cite{LSS02},~\cite{NaNiPoSb03},~\cite{BoSc04,BoScSe05,BL06},~\cite{MiTh},~\cite{AsFrK,AK05},~\cite{BeDPNi05,Ni05,Ni06},~\cite{KRZ13,KRZ15},~\cite{Ak},~\cite{RoRo}). Such unidirectional evolution equations are attracting interest in view of Damage Mechanics as well as  from a purely mathematical viewpoint. Indeed, such problems cannot be classified in the most commonly studied classes of evolution equations due to their unique features. 
%the presence of the positive-part function in the equations and the behavior of solutions, strongly affected by the unidirectionality (see, e.g.,~\cite{AkKi13}), represent often mathematical challenges.

In particular, let us consider %non-negative solutions $w=w(x,t)\geq0$ to
a unidirectional variant of the \emph{porous medium equation with a blow-up term}:
\begin{equation}
\partial_{t}w=\Big(\Delta w^{m}+w^{q}\Big)_{+}\ \text{in }\Omega
\times(0,\infty), \label{pde}%
\end{equation}
where $\Omega$ is a smooth bounded domain of $\mathbb{R}^{d}$, $1<m<\infty$, and $1<q<\infty$. In case $m=q$ and $d=1$, equation \eqref{pde} is proposed in~\cite{BaPr93} as a \emph{damage accumulation model} (see also~\cite{Kachanov})and also mathematically studied in~\cite{BeBi97}, where local (in time) existence of solution is proved and  the  long-time behavior of solutions is  investigated (in particular, regional blow-up phenomena occur for some class of initial data). We also refer the reader to~\cite{NaNiPoSb03,BeDPNi05,Ni05,Ni06,Ak}. In particular, the local existence result of~\cite{BeBi97} is extended for $d \geq 1$ in~\cite{Ak}. One may easily imagine that solutions to \eqref{pde} may blow up in finite time like solutions to equations without non-decreasing constraint. On the other hand,  the behavior  of solutions for small time, i.e., $t \ll 1$, may be strongly influenced by the non-decreasing constraint. Indeed, in case the initial datum $u_0$ fulfills $(\Delta u_0 + u_0^{q/m}) < 0$ in some part of domain, (smooth)  solutions $u(x,t)$ will not evolve immediately and stay as they are for a while. In this view, a sort of free boundary problem with respect to the boundary of the region $R(t) := \{x \in \Omega \colon \Delta u(x,t)+\gamma(u(x,t)) < 0\}$, where the solution $u(x,t)$ does not evolve, is implicitly encoded within equation \eqref{pde}. Therefore, it is not obvious in which regularity class solutions to the initial-boundary value problem for \eqref{pde} can be constructed, for there may arise loss of classical regularity of solutions on the free boundary. 

Equation \eqref{pde} is classified as a fully nonlinear parabolic equation. In general, fully
nonlinear equations are unfit for energy technique. On the other hand, \eqref{pde}
can be transformed into an evolution inclusion of subdifferential type, which
is fitter for energy methods. Indeed, set $u=w^{m}$. Then, \eqref{pde} is
rewritten as
\[
\partial_{t}u^{1/m}=\Big(\Delta u+u^{q/m}\Big)_{+}\ \text{in }\Omega
\times(0,\infty).
\]
Now, applying the (multi-valued) inverse mapping $\alpha(s):=s+\partial
I_{[0,+\infty)}(s)$ of the positive-part function $(s)_{+}$ on both sides, we
deduce that
\begin{equation*}
\partial_{t}u^{1/m}+\partial I_{[0,+\infty)}(\partial_{t}u^{1/m})\ni\Delta
u+u^{q/m}\ \text{\ in }\Omega\times(0,\infty),% \label{pde1.5}%
\end{equation*}
where $\partial I_{[0,+\infty)}$ denotes the subdifferential operator of the
indicator function $I_{[0,+\infty)}:\mathbb{R}\rightarrow\lbrack0,\infty]$
supported over the half-line $[0,+\infty)$, i.e.~for $s \geq 0$,
\begin{align}
\partial I_{[0,+\infty)}(s)  &  =\{\eta\in\mathbb{R}\colon0\geq\eta
(\sigma-s)\text{ for all }\sigma\geq0\}\nonumber\\
&  =
\begin{cases}
\{0\} & \text{if }\ s>0,\\
(-\infty,0] & \text{if }\ s=0.
\end{cases}\label{subid}
\end{align}
Since $s\mapsto s^{1/m}$ is strictly increasing in $(0,+\infty)$ and $u(x,t) \geq 0$ is non-decreasing in time, we observe that
\begin{equation*}
\partial I_{[0,+\infty)}(\partial_{t}u^{1/m})=\partial I_{[0,+\infty
)}(\partial_{t}u)\quad\text{for }(x,t)\in\Omega\times(0,\infty).
%\label{subd}%
\end{equation*}
%
\begin{comment}
 \begin{remark}[Proof of \eqref{subd}]
  {\rm
In case $u>0$ at $(x,t)\in Q$, \emph{\eqref{subd}} follows
obviously, since $\partial_{t}u^{1/m}(x,t)=(1/m)u^{(1-m)/m}(x,t)\partial_{t}u(x,t)$. In case
$u=0$ at $(x,t)\in Q$, we see that $\partial_{t}u(x,\tau)=\partial_{t}%
u^{1/m}(x,\tau)=0$ for $0\leq\tau\leq t$ by the non-negativity $u\geq0$
together with the non-decrease of $u$ in time. Hence, \emph{\eqref{subd}} is
satisfied at any $(x,t)$.
  }
 \end{remark}
\end{comment}
%
\noindent
Thus, \eqref{pde} is reduced to
\begin{equation}
\partial_{t}u^{1/m}+\partial I_{[0,+\infty)}(\partial_{t}u)\ni\Delta
u+u^{q/m}\ \text{\ in }\Omega\times(0,\infty), \label{pde2}%
\end{equation}
which seems to be more tractable with energy techniques, since the
right-hand side  exhibits a  gradient structure.

Equation \eqref{pde2} can be regarded as a \emph{mixed type} of doubly nonlinear
evolution equations, which are extensively studied in the following typical forms:
\[
A(u_{t})+B(u)=0
\]
and
\[
\partial_{t}A(u)+B(u)=0
\]
with two nonlinear operators $A$ and $B$. The former one appears in the study
of generalized Ginzburg-Landau equations (see~\cite{Gurtin} and
also~\cite{A11} with references therein), unidirectional heat flow
(see~\cite{AkKi13}) and so on (see also \cite{Barbu75, Arai, CV, Senba, Colli,
St02, Roubicek, Segatti, MiRo, SSS, S08-BE, MiRoSt08}). The latter one represents nonlinear
diffusion equations, e.g., porous medium/fast diffusion equations and Stefan
problem. On the other hand, \eqref{pde2} is not reduced to such well-studied
classes of doubly nonlinear equations; indeed, it is formulated in the
abstract form,
\begin{equation}
\partial_{t}A_{1}(u)+A_{2}(u_{t})+B(u)=0 \label{mDNE}%
\end{equation}
with nonlinear operators $A_{1}$, $A_{2}$, and $B$. Such a mixed doubly
nonlinear evolution equation has not yet been fully  studied  except
in~\cite{Ak}, where $B$ is  assumed  to be linear (the linearity of $B$ requires $m = q$ in \eqref{pde}) and the linearity plays a crucial role in the analysis of~\cite{Ak}. On the contrary, \eqref{pde2} with $m \neq q$ corresponds to the case where  the three operators are simultaneously nonlinear. Additional difficulties in handling \eqref{pde2} derive from the \emph{unboundedness} of all the operators; indeed, we shall treat \eqref{pde2} in an $L^{2}$-framework, where equations \eqref{pde} and \eqref{pde2} are rigorously equivalent but the corresponding three operators above turn unbounded. In particular, subdifferential operators of indicator functions are essentially unbounded in any function spaces. From these points of view, \eqref{mDNE} is beyond the scope of previous  theory, and therefore, it is worth to develop a new theory to cover \eqref{mDNE}.

The present paper is concerned with the Cauchy-Neumann problem $(P)$:
\begin{alignat}{3}
 \partial_{t}\beta\left(  u\right)  +\xi &  =\LAM\Delta u+\gamma(u)
 \quad && \text{ in } \Omega\times(0,\infty), \label{target eq}\\
\xi &  \in\partial I_{[0,+\infty)}(\partial_{t}u) \quad && \text{ in } \Omega
\times(0,\infty), \label{xi}\\
\partial_{\nu}u  &  =0  && \text{ on } \partial\Omega\times(0,\infty), 
\label{BC}\\
u|_{t = 0} &  =u^{0}\  &&\text{ in } \Omega, \label{IC}%
\end{alignat}
where $\beta$ and $\gamma$ are monotone functions (in $\mathbb{R}$) such that $\beta$ is at most of affine growth and $\beta$ is exactly of $(p-1)$ power growth for $1 < p < +\infty$,
%$\lambda>0$ is a constant
and $\partial_{\nu}$ stands for the outer normal derivative. As is shown above, equation \eqref{target eq} is equivalent to 
\begin{equation*}%\label{cPME}
 \partial_t \beta(u) = \Big( \LAM \Delta u + \gamma(u) \Big)_+ \ \mbox{ in } \Omega \times (0,\infty)
\end{equation*}
and also to a generalized form of \eqref{pde},
$$
 \partial_t w = \Big( \LAM \Delta \beta^{-1}(w) + \gamma \circ \beta^{-1}(w) \Big)_+ \ \mbox{ in } \Omega \times (0,\infty).
 $$
  The main purpose of the present paper is to prove local-in-time existence of solutions for $(P)$. Main result will be stated in Section \ref{S:result}. To prove the existence of local-in-time solutions for ${\normalsize (P)}$, in Section \ref{approx prob existence}, we shall introduce a new approximation for \eqref{target eq},
\begin{alignat}{3}
\partial_{t}\beta\left(  u\right)  +\xi+\mu\gamma(\partial_{t}u)  &
=\LAM\Delta u+\gamma(u)\quad && \text{ if }  \ p\geq2\text{,}  
\label{approx equation}\\
\partial_{t}\beta\left(  u\right)  +\xi+\mu\partial_{t}u  &  =\LAM\Delta u+\gamma(u) \quad && \text{ if }  \ 1<p<2  
\label{approx equation2}%
\end{alignat}
where $\mu$ is a positive parameter. Let us denote by $(P)_{\mu}$ the Cauchy-Neumann problem for \eqref{approx equation} or \eqref{approx equation2} along with \eqref{xi}, \eqref{BC}, and \eqref{IC}. Here, due to the approximation, we can expect existence of \emph{global-in-time} solutions, although the original equation \eqref{target eq} may exhibit blowing-up (in finite time) phenomena. We shall actually prove existence of global-in-time solutions to $(P)_{\mu}$ by introducing a time discretization of \eqref{approx equation} (or \eqref{approx equation2}), by solving a minimization problem at each time step, and finally by passing to the limit as the time step goes to zero. In Section \ref{S:CP}, we next develop a \emph{variational} comparison principle for (possibly non-unique) solutions to $(P)_{\mu}$ (see Proposition \ref{comp princ} below) as well as a comparison principle for sub- and supersolutions to some ODE problem associated with $(P)_{\mu}$ (see Lemma \ref{comp. princ. cost in space sol.} below). It allows us to compare one of solutions to $(P)_{\mu}$ with some constant-in-space supersolutions and subsolutions (see also Lemma \ref{a solution} below). In Section \ref{S:conv}, combining all these facts, we derive some uniform (in $\mu$) estimates for solutions to $(P)_{\mu}$ which enable us to pass to the limit as $\mu\rightarrow0$ and get a local-in-time solution to $(P)$, provided that initial data are uniformly away from $0$, i.e., $u^{0}\geq \delta >0$. The last section of the paper is devoted to proposing a weaker notion of solutions to (P) and to discussing existence of weak solutions also for \emph{non-negative} initial data.

It is noteworthy that the variational comparison principle for (possibly) non-unique solutions is not standard and quite useful in our setting. Indeed, in order to apply classical comparison principles (to derive uniform estimates for approximate solutions), we need, at least, uniqueness of solutions for $(P)_{\mu}$; however,  uniqueness  is not clear due to the severe nonlinearity of the problem. On the other hand, the uniqueness of approximate solutions (for each level $\mu$) is not essential to construct a solution to (P). The variational comparison principle is still applicable to such a setting without paying any extra effort to prove uniqueness.  Difficulties arising from non-uniqueness have already appeared  in~\cite{Ak}, where a (classical) comparison principle is proved for super- and \emph{strictly increasing} subsolutions and some strictly increasing subsolution is constructed by using the linearity of $\gamma(u)$ and further regularity assumption on initial data (see (ii) of Remark \ref{R:Assu} below). The variational comparison principle is one of major and important discoveries of the present paper and plays a crucial role of our analysis.

\bigskip
\noindent
{\bf Notation.} The \emph{positive-part} and \emph{negative-part functions} are given by $(s)_+ := \max\{s,0\}$ and $(s)_- := \max \{-s,0\}$, respectively, for $s \in \R$. Let $X$ be a normed space. We denote by $\|\cdot\|_X$ the norm of $X$ and by $\langle\cdot,\cdot\rangle_X$ the duality pairing between $X$ and its dual space $X^*$. If $X$ is a Hilbert space, we denote by $(\cdot,\cdot)_X$ an inner product in $X$. We denote by $C_w([0,T];X)$ the set of weakly continuous functions on $[0,T]$ with values in $X$. Let $u = u(x,t): \Omega \times [0,\infty) \to \mathbb R$ be a function with space and time variables. Let us also recall the  notion of  \emph{subdifferential operator} $\partial \varphi: X \to X^*$ of a proper (i.e., $\varphi \not\equiv +\infty$) lower semicontinuous convex functional $\varphi : X \to [0,+\infty]$ defined by
\begin{equation}\label{subdif}
 \partial \varphi(u) := \{ \xi \in X^* \colon \varphi(v)-\varphi(u) \geq \langle \xi, v - u \rangle_X  \ \mbox{ for all } v \in X\} 
\end{equation}
 with the domain $D(\partial \varphi) := \{ u \in X \colon \varphi(u) < +\infty, \ \partial \varphi(u) \neq \emptyset\}$. In case $X$ is a Hilbert space, the duality pairing in \eqref{subdif} may be replaced with the inner product $(\cdot,\cdot)_X$. Throughout the paper, for each $t \geq 0$ fixed, we simply denote by $u(t)$ the function $u(\cdot,t) : \Omega \to \mathbb R$ with only the space variable. %For $T > 0$, $C_w([0,T];X)$ and $C_+([0,T];X)$ stand for the sets of weakly- and right- continuous functions on $[0,T]$ with values in a normed space $X$, respectively. We denote by $\|\cdot\|_p$, $1 \leq p \leq \infty$ the $L^p(\Omega)$-norm, that is, $\|f\|_p := (\int_\Omega |f(x)|^p \, \d x)^{1/p}$ for $p \in [1,\infty)$ and $\|f\|_\infty := \esssup_{x \in \Omega} |f(x)|$. Denote also by $(\cdot,\cdot)$ the $L^2$-inner product, i.e., $(u,v) := \int_\Omega u(x) v(x) \, \d x$ for $u,v \in L^2(\Omega)$. For each normed space $X$ and $T > 0$, $C_w([0,T];X)$ denotes the space of weakly continuous functions on $[0,T]$ with values in $X$. We also simply write $u(t)$ instead of $u(\cdot,t)$, which is regarded as a function from $\Omega$ to $\mathbb R$, for each fixed $t \geq 0$.
Moreover, we denote by $I_{[0,+\infty)}$ the indicator function supported on the half-line $[0,+\infty)$. Let $I: L^2(\Omega) \to [0,+\infty]$ be the indicator function supported on the closed convex set $K := \{u \in L^2(\Omega) \colon u \geq 0 \ \mbox{ a.e.~in } \Omega\}$, that is,
\begin{equation}\label{A}
I(u) = \begin{cases}
	  0 &\mbox{ if } \ u \in K,\\
	  \infty &\mbox{ otherwise }
	 \end{cases}
	 \quad \mbox{ for } \ u \in L^2(\Omega).
\end{equation}
	 Moreover, let $\partial I_{[0,+\infty)}$ also denote the subdifferential operator (precisely, $\partial_{\mathbb R} I_{[0,+\infty)}$) in $\mathbb R$ (see~\eqref{subid}) as well as that (precisely, $\partial_{L^2(\Omega)} I$) in $L^2(\Omega)$ defined by \eqref{subdif}, that is,
	 $$
	 \partial_{L^2(\Omega)} I(u) = \left\{ \eta \in L^2(\Omega) \colon (\eta, u - v)_{L^2(\Omega)} \geq 0 \ \mbox{ for all } \ v \in K \right\} \quad \mbox{ for } \ u \in K.
	 $$
Here, we note that these two notions of subdifferentials are equivalent each other in the following sense: for $u, \eta \in L^2(\Omega)$,
	 $$
	 \eta \in \partial_{L^2(\Omega)} I_{[0,+\infty)}(u) \quad \mbox{ if and only if } \quad \eta(x) \in \partial_{\mathbb R} I_{[0,+\infty)}(u(x)) \ \mbox{ a.e.~in } \Omega
	 $$
	 (see, e.g.,~\cite{HB1,HB3}). 
We denote by $C$ a non-negative constant, which does not depend on the elements of the corresponding space or set and may vary from line to line. In order to emphasize some dependence of such a constant on some variable, e.g., $\sigma$, we may write $C_\sigma$, which may also vary from place to place.

\section{Assumptions and main result}\label{S:result}

Let us start by enlisting our assumptions.

\begin{description}
 \item[(A1)] There exist positive constants $C_1$, $C_2$ and an exponent $p$ satisfying
	    $$
	    1 < p < \dfrac{2d}{(d-2)_+}
	    $$
	    and a monotone continuous function $\gamma:[0,+\infty)\rightarrow\lbrack0,+\infty)$ satisfying
	    \begin{equation*}%\label{gamma-zero}
	    \gamma(s)=0 \ \mbox{ if and only if } \ s=0
	    \end{equation*}
	    such that
\begin{alignat}{3}
C_{1}\left(  |s|^{p}-1\right) & \leq\hat{\gamma}(s) \ &\mbox{ for all } \ s \geq 0,\label{gamma1}\\
|\gamma(s)|^{p^{\prime}}  &  \leq C_2(|s|^{p}+1)\ &\mbox{ for all } \ s \geq 0,\label{gamma4}%
\end{alignat}
where $\hat{\gamma}(s):=\int_{0}^{s}\gamma(r)\,\mathrm{d}r$ and $p' := p/(p-1)$.

 \item[(A2)] It holds that $u^{0}\in H^{2}(\Omega)\cap L^{\infty}(\Omega)$ and $\partial_\nu u^0 = 0$ a.e.~in $\partial \Omega$. Moreover, there
exists $\delta>0$ such that $u^{0}\geq\delta$ a.e.~in $\Omega$.% and
%\begin{equation}
%\left(  \LAM\Delta u^{0}+\gamma(u^{0})\right)_{-} \in L^\infty(\Omega), \label{i d assumption}%
%\end{equation}
%where $(s)_- := \max \{ -s , 0\} \geq 0$. 

 \item[(A3)] It holds that $\beta\in C^{1}((0,+\infty))$ and $\beta^{\prime}$ is non-increasing and positive in $(0,+\infty)$. In particular, for any $\vep > 0$, there exists a constant $C_{\vep}$ ($=\beta'(\vep)$) such that
	    \begin{equation}
	     0<\beta^{\prime}(s)\leq C_{\vep} \quad\text{ for all } \ s\geq\vep.
\label{beta growth}%
	    \end{equation}
	    
% \item[(A4)] For each $\mu > 0$ small enough, there exists $\tau_\mu > 0$ such that, for any $0 < \tau < \tau_\mu$ and $0 < a < \infty$, the function
%\begin{equation}\label{a4-func}
%  \mu\gamma\left(  \frac{\cdot-a}{\tau}\right)  +\frac{\beta(\cdot)-\beta
%(a)}{\tau}-\gamma(\cdot)
%\end{equation}
%has a finite number of zeros in $[a,+\infty)$.
  
\end{description}

\begin{remark}[Assumptions]\label{R:Assu}
 {\rm
 \begin{enumerate}
  \item[(i)] Assumption {\bf (A1)} also implies
\begin{alignat}{3}
 \gamma(s)s & \geq C_1\left(  |s|^{p}-1\right) &&\mbox{ for all } s \geq 0,\label{gamma2}\\
 \hat{\gamma}(s)&\leq C_3(|s|^{p}+1) \quad &&\mbox{ for all } s \geq 0\label{gamma1.5}
\end{alignat}
	     for some $C_3 \geq 0$. Indeed, we observe that $\gamma(s)s \geq \hat \gamma(s)$ by convexity of $\hat \gamma$.
  \item[(ii)] {\bf (A2)} is weaker than assumptions on initial data in~\cite{Ak}. Indeed, it is no longer necessary to suppose that $(\Delta u^0+\gamma(u^0))_-$ belongs to $L^\infty(\Omega)$, thanks to variational comparison principle developed in \S \ref{S:CP} and an improved argument to derive a uniform estimate for approximate solutions.
  \item[(iii)] Assumption $u^{0}\geq \delta$ in {\bf (A2)} will be used to derive by {\bf (A3)} the boundedness of 
$\beta^{\prime}(u)$, and hence, of $\partial_{t}\beta(u)$ (in some suitable norms). 
In Section \ref{sec deg data}, we propose a weaker notion of solution  involving neither $\beta^{\prime}(u)$ nor $\partial_{t}\beta(u)$. 
This allows us to prove existence of weak solutions to $(P)$ without assuming 
$u^{0}\geq \delta$.
  \item[(iv)] Let $\vep > 0$ be fixed. From assumption {\bf (A3)}, it follows that
	       \begin{equation}
		|\beta(s)|\leq C(\vep)(s+1) \quad \mbox{ for all } \ s\geq\vep
		 \label{bound b}%
	       \end{equation}
	      for some constant $C(\vep)$ depending on $\vep$. In particular, for $q \in [1,+\infty)$, the mapping $u \mapsto \beta(u(\cdot))$ is strongly continuous in $L^q(\Omega)$, provided that $u \geq \vep$ for a.e.~in $\Omega$.
 \end{enumerate}
 }
\end{remark}

 \begin{remark}[Examples of $\gamma$ and $\beta$]\label{R:ex}
  {\rm
The above assumptions are satisfied by, e.g., 
\[
  \gamma(s)=|s|^{p-2}s \ \mbox{ for } \ p>1 \ \mbox{ and } \ 
  \beta(s) =\begin{cases}
	  \frac{s^{1-\alpha}}{1-\alpha} & \text{if } \ \alpha \geq 0, \ \alpha \neq 1\text{,}\\
	  \log s & \text{if } \ \alpha=1,
	    \end{cases}
\]
where the choice of $\beta$ above also stems from a damage accumulation model studied in~\cite{BaPr93,BeBi97,Ak}. 
In particular, it allows us to apply the main result stated below to \eqref{pde} under
  $$
  1 \leq m < +\infty \quad \mbox{ and } \quad 1 < \frac q m < 2^*.
  $$
  Furthermore, the case $\alpha = 1$ corresponds to the equation,
  $$
  \partial_t w = \left( \LAM \Delta e^w + e^{pw}\right)_+ \ \mbox{ in } \Omega \times (0,+\infty), \quad 1 < p < 2^*,
  $$
  and the case $\alpha > 1$ includes the so-called \emph{Penrose-Fife} equation (i.e., $\alpha = 2$).
}
 \end{remark}

Before stating our main result, let us give a definition of \emph{strong solution} to
the initial-boundary value problems $(P)$ and $(P)_{\mu}$ in a precise way.
\begin{definition}[Strong solution] Let $T>0$ and $\mu\geq0$. A positive function $u : \Omega \times (0,T) \to (0,+\infty)$ is called a \emph{strong solution} of $(P)_{\mu}$ {\rm (}in particular, $(P)$ if $\mu=0${\rm )} on $[0,T]$, if the following {\rm (i)--(iii)} are all satisfied\/{\rm :}
\begin{enumerate}
 \item[\rm (i)] It holds that
	      \begin{gather*}
	       u \in W^{1,2}(0,T;L^{2}(\Omega))\cap C_{w}([0,T];H^{1}(\Omega)), \\
	      \beta(u) \in W^{1,2}(0,T;L^{2}(\Omega)), \quad \Delta u \in L^{2}(0,T;L^{2}(\Omega)).
	      \end{gather*}
%	    where
%	    $$
%	    \tilde{p} := \begin{cases}
%		       \max\{2,p\} & \text{if } \ \mu>0,\\
%		       2 & \text{if } \ \mu=0
%			 \end{cases}
%	      \quad \mbox{ and } \quad \tilde{p}^{\prime}:= \dfrac{\tilde{p}}{\tilde{p}-1}.
%	    $$
 \item[\rm (ii)] It is fulfilled that $u(\cdot,t)\in H^{2}(\Omega)$ for a.e.~$t\in (0,T)$ and $\partial_{\nu}u(t) =0$ a.e.~on $\partial\Omega$. Moreover, $\partial_{t}u(x,t) \geq0$ for a.e.~$(x,t)\in\Omega\times (0,T)$.
 \item[\rm (iii)] There exists $\xi\in L^{\infty}(0,T;L^2(\Omega))$ such that the following relations hold true\/{\rm :}
\begin{alignat*}{3}
\xi & \in\partial I_{[0,+\infty)}(\partial_{t}u) \ &\text{ a.e.~in }\Omega
\times(0,T),&\\
0  &  =\mu\gamma(\partial_{t}u)+\partial_{t}\beta(u)+\xi-\LAM\Delta
u-\gamma(u) \ &\text{ a.e.~in }\Omega\times(0,T) &\quad \text{ if } \ p\geq2\text{,}\\
0  &  =\mu\partial_{t}u+\partial_{t}\beta(u)+\xi-\LAM\Delta
u-\gamma(u) \ &\text{ a.e.~in }\Omega\times(0,T) &\quad \text{ if } \ 1<p<2.
\end{alignat*}
\end{enumerate}
Furthermore, $u$ is called a strong solution on $[0,T)$, if so is $u$ on $[0,S]$ for any $0 < S < T$.
\end{definition}

Our main result reads as follows.

\begin{theorem}[Existence of local solutions to $(P)$]\label{main thm}
 Suppose {\bf (A1)-(A3)} are all satisfied. Then there exist $T_0>0$ and a strong solution $u$ of $(P)$ defined on $[0,T_0]$. In addition, if $(\Delta u^0 + \gamma(u^0))_- \in L^q(\Omega)$ for $q \in (2,+\infty)$ {\rm (}$q = \infty$, respectively{\rm )}, then $\xi \in L^\infty(0,T;L^q(\Omega))$ {\rm (}$\xi \in L^\infty(\Omega \times (0,T_0))$, respectively{\rm )}.

 Set
 $$
 T_{\max} := \sup \left\{ S > 0 \colon (P) \mbox{ admits a strong solution on } [0,S] \right\} \geq T_0 > 0.
 $$
 Then
 $$
 T_{\max} \geq \hat T(M) := \int^{+\infty}_{\beta(M)} \dfrac{\d b}{\gamma \circ \beta^{-1}(b)} > 0 \quad \mbox{ with } \ M := \|u_0\|_{L^\infty(\Omega)}.
 $$
In particular, if $T_{\max}$ is finite, then
	       $$
	       T_{\max}\leq\hat{T}(\delta)%=\int_{\delta}^{+\infty}\beta^{\prime}(r)/\gamma(r)\mathrm{d}r
	       = \int^{+\infty}_{\beta(\delta)} \dfrac{\d b}{\gamma \circ \beta^{-1} (b)}.
	       $$
\end{theorem}

\begin{remark}[Global existence]
 {\rm
Thanks to Theorem \ref{main thm}, if there exists $C>0$ such that
 \begin{equation}
  \gamma\circ\beta^{-1}(s) \leq C(s+1) \quad \text{ for all } \ s > 0\text{,}
   \label{sublinear growth}%
 \end{equation}
 then $T_{\max} = +\infty$.
 }
\end{remark}

Theorem \ref{main thm} will be proved step by step in the next three sections. 

\section{Global solutions for approximate problems\label{approx prob existence}}

We only treat the case $p\geq2$ and prove existence of solutions to (\ref{approx equation}), (\ref{xi})-(\ref{IC}). The other case $1 <p<2$ is simpler; indeed, $\gamma(u)$ is sublinear, and therefore, as in~\cite{Ak}, it is enough to add a linear regularization to construct global (approximate) solutions (hence, the corresponding approximate equation is of the form \eqref{approx equation2}). We can also handle \eqref{approx equation2} by repeating the same argument with minor modifications (see Remark \ref{p small} and \eqref{h1 est} below). The main result of this section is stated as follows:

 \begin{proposition}[Existence of global solutions to $(P)_{\mu}$]\label{P:aprx}
  Let assumptions \emph{\textbf{(A1)-(A3)}} be satisfied. Then, for every $T>0$, there exists a
  strong solution $u_\mu$ to $(P)_{\mu}$ on $[0,T]$. In particular, if $p \geq 2$, then $u_\mu$ further fulfills
  $$
  u_\mu \in W^{1,p}(0,T;L^p(\Omega)), \quad \Delta u_\mu \in L^{p'}(0,T;L^{p'}(\Omega)), \quad p':= p/(p-1).
  $$
\end{proposition}

%We follow the outline of proof for \cite[Theorem 3.8]{Ak}:
We first introduce a time discretization of $(P)_{\mu}$ and rewrite the discretized equation
at each time step as a (non-convex) minimization problem. Secondly, we establish
uniform estimates for the discretized solutions and then pass to the limit as the time step goes to $0$.

\subsection{Discretization.}

Let $N\in\mathbb{N} $ and $\tau=T/N$ and consider the discretized problem
$(P)_{\mu,n}$: Find $\{u_{n}\}\in\left(  H^{1}(\Omega)\right)^{N}$ such that%
\begin{align}
\mu\gamma\left(  \frac{u_{n+1}-u_{n}}{\tau}\right)  +\frac{\beta
(u_{n+1})-\beta(u_{n})}{\tau}+\xi_{n+1}  &  =\LAM\Delta u_{n+1}%
 +\gamma(u_{n+1})%\text{ in }L^{p^{\prime}}(\Omega)
 \text{,}%
\label{discrete euler}\\
\xi_{n+1}  &  \in\partial I_{[0,+\infty)}\left(  \frac{u_{n+1}-u_{n}}{\tau
 }\right)%  \text{ in }L^{p^{\prime}}(\Omega)
 \text{,}\label{discrete euler2}\\
u_{0}  &  =u^{0}\text{,} \label{discrete euler 3}%
\end{align}
for every $n\in\{0,\dots,N-1\}$. To solve it, define a functional $J_{n+1}:H^{1}%
(\Omega)\rightarrow(-\infty,+\infty]$ by
\begin{align}
J_{n+1}(u)  &  =\mu\int_{\Omega}\tau\hat{\gamma}\left(  \frac{u-u_{n}}{\tau
}\right)  +\frac{1}{\tau}\int_{\Omega}\hat{\beta}(u)+\tau I\left(
\frac{u-u_{n}}{\tau}\right) \nonumber\\
&  \quad +\frac{\LAMuno}{2}\left\Vert \nabla u\right\Vert _{L^{2}(\Omega
)}^{2}-\int_{\Omega}\hat{\gamma}(u)-\left(  \frac{\beta(u_{n})}{\tau
},u\right)  _{L^{2}(\Omega)} \label{J functional}%
\end{align}
where $\hat{\beta}(s)=\int_{0}^{s}\beta(r)\mathrm{d}r$, $\hat{\gamma}%
(s)=\int_{0}^{s}\gamma(r)\mathrm{d}r$ and $I(u)$ is defined by \eqref{A}.
Let us prove existence of a\ minimizer of $J_{n+1}$. In order to
do this, we first show that for $\tau>0$ sufficiently small
\begin{align}
\lefteqn{\mu\int_{\Omega}\tau\hat{\gamma}\left(  \frac{u-u_{n}}{\tau}\right)
 -\int_{\Omega}\hat{\gamma}(u)
 }\nonumber\\
 &\geq-C_{p,\tau,\mu,u_{n}}+\frac{\mu}{2}%
\int_{\Omega}\tau\hat{\gamma}\left(  \frac{u-u_{n}}{\tau}\right)  \text{ for
all }u\in H^{1}(\Omega)\text{,} \label{computation1}%
\end{align}
where $C_{p,\tau,\mu,u_{n}}$ is a positive constant depending on $p,\tau,\mu,$
and $u_{n}$. Indeed, thanks to assumption \eqref{gamma1.5},
\begin{align*}
\frac{1}{2}\tau\hat{\gamma}\left(  \frac{u-u_{n}}{\tau}\right)   &  \geq
C\tau\left\vert \frac{u-u_{n}}{\tau}\right\vert ^{p}-C\geq\tau C\left\vert
\frac{u}{\tau}\right\vert ^{p}-C_{p,\tau,u_{n}}\\
&  \geq C\tau^{1-p}\hat{\gamma}(u)-C_{p,\tau,u_{n}}\text{.}%
\end{align*}
Here and henceforth the symbol $C$ will denote a positive constant independent of
$\mu$, $n$, $\tau$ and possibly varying from line to line. If $\tau$ is small
enough, namely $C\tau^{1-p}\geq\frac{1}{\mu}$, we get
\[
\frac{\mu}{2}\int_{\Omega}\tau\hat{\gamma}\left(  \frac{u-u_{n}}{\tau}\right)
-\int_{\Omega}\hat{\gamma}(u)\geq-C_{p,\tau,\mu,u_{n}},%
\]
which implies inequality (\ref{computation1}). This proves the functional
$J_{n+1}(\cdot)$ to be bounded from below and coercive in $H^1(\Omega)$ for every $u_{n}$ fixed. Moreover,
we decompose
\[
J_{n+1}(u)=\frac{1}{\tau}\int_{\Omega}\hat{\beta}(u)+\tau I\left(
\frac{u-u_{n}}{\tau}\right)  +\frac{\LAMuno}{2}\left\Vert \nabla
u\right\Vert _{L^{2}(\Omega)}^{2}-\left(  \frac{\beta(u_{n})}{\tau},u\right)
_{L^{2}(\Omega)}+G_{n+1}(u)\text{,}%
\]
where
\[
G_{n+1}(u)=\mu\int_{\Omega}\tau\hat{\gamma}\left(  \frac{u-u_{n}%
}{\tau}\right)  -\int_{\Omega}\hat{\gamma}(u)\text{.}%
\]
Note that $J_{n+1}(\cdot)-G_{n+1}(\cdot)$ is convex and lower semicontinuous
in $H^{1}(\Omega)$. Let $\{u_{k}\}$ be a minimizing sequence for $J_{n+1}$.
Then, $\{u_{k}\}$ is bounded in $H^{1}(\Omega)$ from the coercivity of $J_{n+1}$. By virtue of the Sobolev
embedding results, there exists a (not relabeled) subsequence $u_{k}%
\rightarrow u$ in $L^{p}(\Omega)$ by $p<2^{\ast}$. Thanks to continuity of
$\gamma$ and assumption (\ref{gamma1.5}), we have
\[
\lim_{k\rightarrow\infty}G_{n+1}(u_{k})=G_{n+1}(u)\text{.}%
\]
Thus, by lower semicontinuity,
\begin{align*}
\inf_{w\in H^1(\Omega)} J_{n+1}(w)  &  =\liminf_{k\rightarrow\infty}J_{n+1}(u_{k})\geq
\liminf_{k\rightarrow\infty}\left(  J_{n+1}(u_{k})-G_{n+1}(u_{k})\right)
+\lim_{k\rightarrow\infty}G_{n+1}(u_{k})\\
&  \geq J_{n+1}(u)-G_{n+1}(u)+G_{n+1}(u)=J_{n+1}(u)\text{.}%
\end{align*}
Thus $J_{n+1}$ admits at least one minimizer
$u_{n+1}\in H^{1}(\Omega)$. Moreover, $u_{n+1}$ solves the Euler-Lagrange
equation (\ref{discrete euler})-(\ref{discrete euler2}) in $\left(
H^{1}(\Omega)\right)  ^{\ast}$, namely
\begin{align}
\mu\gamma\left(  \frac{u_{n+1}-u_{n}}{\tau}\right)  +\frac{\beta
(u_{n+1})-\beta(u_{n})}{\tau}+\xi_{n+1}  &  =\LAM\Delta u_{n+1}%
+\gamma(u_{n+1})%
\text{,}\label{eq h1 star}\\
\xi_{n+1}  &  \in\partial_{H^{1}(\Omega)}I\left(  \frac{u_{n+1}-u_{n}}{\tau
}\right),\nonumber
\end{align}
where $\partial_{H^1(\Omega)}$ stands for the subdifferential in $H^1(\Omega)$. Here we also used the fact that $\int_\Omega \hat \gamma(u)$, $\int_\Omega \hat \beta(u)$ are of class $C^1$ in $H^1(\Omega)$. Indeed, a sum rule for subdifferential is nontrivial, but it holds at least for the sum of smooth (e.g.~$C^1$) functionals and a non-smooth (e.g.~convex, l.s.c.) functional. Invoking a regularity theory for variational inequalities of
obstacle type (see, e.g., \cite{AkKi13}), as in \cite[Theorem 2.1]{Ak}, we can prove that
$$
u_{n+1}\in D(-\Delta) := \{w \in L^{p'}(\Omega) \colon \Delta w \in L^{p'}(\Omega) \mbox{ and } \partial_\nu w = 0 \mbox{ a.e.~in } \partial \Omega \}
$$
for $n = 0,1,\ldots,N-1$. Here we also used the fact that $u_0 = u^0 \in D(-\Delta)$ (see {\bf (A2)}) to apply the regularity theory. Hence, by comparison in (\ref{eq h1 star}), $\xi_{n+1}\in L^{p^{\prime}}(\Omega)$. In particular,%
\[
\xi_{n+1}\in\partial I_{[0,+\infty)}\left(  \frac{u_{n+1}-u_{n}}{\tau}\right)
\text{ a.e.~in }\Omega
\]
and $u_{n+1}$ solves (\ref{discrete euler})-(\ref{discrete euler 3}) in $L^{p'}(\Omega)$.
Furthermore, as $J_{n+1}(u_{n+1})$ is finite, we have $I_{[0,+\infty)}\left(
\frac{u_{n+1}-u_{n}}{\tau}\right)  <+\infty$ a.e.~in $\Omega$. Thus,
$u_{n+1}-u_{n}$ must be non-negative. Recalling $u^{0}\geq\delta$ we have,
$u_{n+1}\geq u_{n}\geq\delta$ for every $n\in\{0,\dots,N-1\}$.

\subsection{Uniform estimates.}

We now derive some a-priori estimates for $u_{n}$. Testing equation
(\ref{discrete euler}) by $\frac{u_{n+1}-u_{n}}{\tau}$, we get
\begin{gather}
\mu\int_{\Omega}\gamma\left(  \frac{u_{n+1}-u_{n}}{\tau}\right)  \frac
{u_{n+1}-u_{n}}{\tau}+\left(  \frac{\beta(u_{n+1})-\beta(u_{n})}{\tau}%
,\frac{u_{n+1}-u_{n}}{\tau}\right)  _{L^{2}(\Omega)}\nonumber\\
+\left(  \xi_{n+1},\frac{u_{n+1}-u_{n}}{\tau}\right)  _{L^{p}(\Omega
)}+\frac{\LAMuno}{\tau}\left\Vert \nabla u_{n+1}\right\Vert _{L^{2}(\Omega
)}^{2}-\frac{\LAMuno}{\tau}\int_{\Omega}\nabla u_{n+1}\cdot\nabla
 u_{n}\nonumber\\
 =\int_{\Omega}\gamma(u_{n+1})\frac{u_{n+1}-u_{n}}{\tau}\text{.}
\label{discrete initial estimate}%
\end{gather}
Note that $\left(  \xi_{n+1},\frac{u_{n+1}-u_{n}}{\tau}\right)  _{L^{p}%
(\Omega)}=0$. Indeed, either $\xi_{n+1}(x)=0$ or $u_{n+1}(x)=u_{n}\left(
x\right)$ holds for a.e.~$x\in\Omega$. By monotonicity of $\beta$,
\[
\left(  \frac{\beta(u_{n+1})-\beta(u_{n})}{\tau},\frac{u_{n+1}-u_{n}}{\tau
}\right)  _{L^{2}(\Omega)}\geq0.
\]
By assumption (\ref{gamma2}), we have%
\[
\mu\int_{\Omega}\gamma\left(  \frac{u_{n+1}-u_{n}}{\tau}\right)  \frac
{u_{n+1}-u_{n}}{\tau}\geq\mu C_1\left\Vert \frac{u_{n+1}-u_{n}}{\tau}\right\Vert
_{L^{p}(\Omega)}^{p}-\mu C_1\text{.}%
\]
By using (\ref{gamma4}) and the Young inequality, for any $\alpha > 0$, one can take constants $C_\alpha > 0$ (which may vary from line to line below) such that
\begin{align*}
\int_{\Omega}\gamma(u_{n+1})\frac{u_{n+1}-u_{n}}{\tau}  &  \leq\alpha
\left\Vert \frac{u_{n+1}-u_{n}}{\tau}\right\Vert _{L^{p}(\Omega)}
^{p}+C_{\alpha}\left\| \gamma(u_{n})\right\|_{L^{p'}(\Omega)}^{p'}\\
&  \leq\alpha\left\Vert \frac{u_{n+1}-u_{n}}{\tau}\right\Vert _{L^{p}(\Omega
)}^{p}+C_{\alpha}\left\Vert u_{n}\right\Vert _{L^{p}(\Omega)}^{p}+C_{\alpha}.
\end{align*}
Moreover, we estimate
\[
-\int_{\Omega}\nabla u_{n+1} \cdot \nabla u_{n}\geq-\frac{1}{2}\left\Vert \nabla
u_{n+1}\right\Vert _{L^{2}(\Omega)}^{2}-\frac{1}{2}\left\Vert \nabla
u_{n}\right\Vert _{L^{2}(\Omega)}^{2}.%
\]
Substituting all these facts into (\ref{discrete initial estimate}), we get%
\begin{align*}
&  \mu C\left\Vert \frac{u_{n+1}-u_{n}}{\tau}\right\Vert _{L^{p}(\Omega)}%
^{p}+\frac{\LAMuno}{2\tau}\left\Vert \nabla u_{n+1}\right\Vert
_{L^{2}(\Omega)}^{2}-\frac{\LAMuno}{2\tau}\left\Vert \nabla u_{n}%
\right\Vert _{L^{2}(\Omega)}^{2}\\
&  \leq\alpha\left\Vert \frac{u_{n+1}-u_{n}}{\tau}\right\Vert _{L^{p}(\Omega
)}^{p}+C_{\alpha}\left\Vert u_{n}\right\Vert _{L^{p}(\Omega)}^{p}+C_{\alpha
} + \mu C_1 \text{,}%
\end{align*}
which yields, for $\alpha > 0$ small enough,%
\begin{align}
&  \frac{\mu}{2}C_1\left\Vert \frac{u_{n+1}-u_{n}}{\tau}\right\Vert
_{L^{p}(\Omega)}^{p}+\frac{\LAMuno}{2\tau}\left\Vert \nabla u_{n+1}%
\right\Vert _{L^{2}(\Omega)}^{2}-\frac{\LAMuno}{2\tau}\left\Vert \nabla
u_{n}\right\Vert _{L^{2}(\Omega)}^{2}\nonumber\\
&  \leq C_{\mu}\left\Vert u_{n}\right\Vert _{L^{p}(\Omega)}^{p}+C_{\mu
}\text{.} \label{disc intermediate estimate}%
\end{align}
Note that
\[
\frac{u_{n+1}^{p}-u_{n}^{p}}{\tau}=p\tilde{u}_{n}^{p-1}\frac{u_{n+1}-u_{n}%
}{\tau}%
\]
for some $\tilde{u}_{n}\in\lbrack u_{n},u_{n+1}].$ Hence, by using the Young
inequality, for $\vep > 0$, there is some $C_\vep > 0$ such that
\[
\int_{\Omega}\frac{u_{n+1}^{p}-u_{n}^{p}}{\tau}\leq \vep \int_{\Omega}\left(
\frac{u_{n+1}-u_{n}}{\tau}\right)  ^{p}+ C_\vep \int_{\Omega}\tilde{u}_{n}^{p}.
\]
Substituting it into (\ref{disc intermediate estimate}) and recalling
$u_{n+1}\geq\tilde{u}_{n}$, we obtain
\[
\mu\int_{\Omega}\frac{u_{n+1}^{p}-u_{n}^{p}}{\tau}+
%C\left\{
\frac{\LAMuno}{2\tau}\left\Vert \nabla u_{n+1}\right\Vert _{L^{2}(\Omega)}^{2}%
-\frac{\LAMuno}{2\tau}\left\Vert \nabla u_{n}\right\Vert _{L^{2}(\Omega)}^{2}
%\right\}
\leq C_{\mu}\left\Vert u_{n+1}\right\Vert _{L^{p}(\Omega)}%
^{p}+C_{\mu}\text{.}%
\]
Multiplying both sides by $\tau$ and taking the sum over $\{0,...,n\}$, we
get,
\begin{align}
 \mu\left\Vert u_{n+1}\right\Vert _{L^{p}(\Omega)}^{p}
 +%C
 \frac{\LAMuno}%
{2}\left\Vert \nabla u_{n+1}\right\Vert _{L^{2}(\Omega)}^{2}  &  \leq
 %C
 \frac{\LAMuno}{2}\left\Vert \nabla u^{0}\right\Vert _{L^{2}(\Omega)}%
^{2}+\mu\left\Vert u^{0}\right\Vert _{L^{p}(\Omega)}^{p}\nonumber\\
& \quad +C_{\mu}n\tau+\tau C_{\mu}\sum_{k=0}^{n}\left\Vert u_{k+1}\right\Vert
_{L^{p}(\Omega)}^{p}\nonumber\\
 &  \leq %C
 \frac{\LAMuno}{2}\left\Vert \nabla u^{0}\right\Vert
_{L^{2}(\Omega)}^{2}+\mu\left\Vert u^{0}\right\Vert _{L^{p}(\Omega)}%
^{p} +C_{\mu}T \nonumber\\
&\quad  +\tau C_{\mu}\sum_{k=1}^{n}\left\Vert u_{k}\right\Vert
_{L^{p}(\Omega)}^{p}+\tau C_{\mu}\left\Vert u_{n+1}\right\Vert _{L^{p}%
(\Omega)}^{p}\text{.} \label{formula}%
\end{align}
In particular, for $\tau<\mu/2C_{\mu}$,
\begin{align*}
\left(  \mu-\tau C_{\mu}\right)  \left\Vert u_{n+1}\right\Vert _{L^{p}%
 (\Omega)}^{p}  &  \leq %C
 \frac{\LAMuno}{2}\left\Vert \nabla u^{0}%
\right\Vert _{L^{2}(\Omega)}^{2}+\mu\left\Vert u^{0}\right\Vert _{L^{p}%
(\Omega)}^{p}\\
& \quad +C_{\mu}T+\tau C_{\mu}\sum_{k=1}^{n}\left\Vert u_{k}\right\Vert
_{L^{p}(\Omega)}^{p}\text{.}%
\end{align*}
Multiplying by $\left(  \mu-\tau C_{\mu}\right)  ^{-1}$, we have
\[
\left\Vert u_{n+1}\right\Vert _{L^{p}(\Omega)}^{p}\leq C_{\mu}\LAM\left\Vert \nabla u^{0}\right\Vert _{L^{2}(\Omega)}^{2}+2\left\Vert
u^{0}\right\Vert _{L^{p}(\Omega)}^{p}+2C_{\mu}T+2\tau C_{\mu}\sum_{k=1}%
^{n}\left\Vert u_{k}\right\Vert _{L^{p}(\Omega)}^{p}\text{.}%
\]
By applying the discrete Gronwall lemma, one has
\[
\max_{n}\left\Vert u_{n}\right\Vert _{L^{p}(\Omega)}^{p}\leq C_{\mu}\left(
\left\Vert u^{0}\right\Vert _{L^{p}(\Omega)}^{p}+\left\Vert \nabla
u^{0}\right\Vert _{L^{2}(\Omega)}^{2}+ T\right)
\]
and hence, substituting the above into (\ref{formula}), we also infer that
\begin{equation}
 \max_{n} %\left\{  \left\Vert u_{n}\right\Vert _{L^{p}(\Omega)}^{p}+
	    \left\Vert\nabla u_{n}\right\Vert _{L^{2}(\Omega)}^{2}%\right\}
 \leq C_{\mu}\left(
\left\Vert u^{0}\right\Vert _{L^{p}(\Omega)}^{p}+\left\Vert \nabla
u^{0}\right\Vert _{L^{2}(\Omega)}^{2}+ T \right)  \text{.}
\label{discrete unif est1}%
\end{equation}
By (\ref{disc intermediate estimate}), we have
\begin{equation}
\sum_{n=0}^{N-1}\tau\left\Vert \frac{u_{n+1}-u_{n}}{\tau}\right\Vert
_{L^{p}(\Omega)}^{p}\leq C_{\mu}\text{.} \label{lp integral est}%
\end{equation}
In case $1 < p < 2$, \eqref{lp integral est} with $p = 2$ directly follows. We claim that
\begin{equation}
\max_{n}\left\Vert u_{n}\right\Vert _{H^{1}(\Omega)}\leq C_{\mu}\text{.}
\label{h1 est}%
\end{equation}
Indeed, this is trivially true if $p\geq2$. If $p<2$, as a consequence of the
Gagliardo-Nirenberg interpolation inequality, we have
\[
\left\Vert u_{n}\right\Vert _{L^{2}(\Omega)}\leq C\left\Vert \nabla
u_{n}\right\Vert _{L^{2}(\Omega)}^{a}\left\Vert u_{n}\right\Vert
_{L^{p}(\Omega)}^{1-a}%
\]
for some $a=a(p,d)\in(0,1)$, and thus, (\ref{h1 est}) follows.

Thanks to estimates (\ref{discrete unif est1}), (\ref{lp integral est}) and
assumption (\ref{gamma4}) (see also Remark \ref{R:Assu}), we have
\begin{align}
\max_{n}\left\Vert \gamma(u_{n})\right\Vert _{L^{p\prime}(\Omega)}^{p^{\prime
}}  &  \leq C_{\mu}\text{,}\label{gamma est}\\
\sum_{n=0}^{N-1}\tau\left\Vert \gamma\left(  \frac{u_{n+1}-u_{n}}{\tau
}\right)  \right\Vert _{L^{p^{\prime}}(\Omega)}^{p^{\prime}}  &  \leq C_{\mu
}\text{.} \label{gamma est2}%
\end{align}
Thanks to \eqref{bound b} and estimate (\ref{discrete unif est1}), we have
\begin{align*}
\max_{n}\left\Vert \beta(u_{n})\right\Vert _{L^{p}(\Omega)}^{p}  &  \leq
C_{\mu}\text{,}\\
\max_{n}\left\Vert \beta(u_{n})\right\Vert _{L^{2}(\Omega)}^{2}  &  \leq
C_{\mu}.
\end{align*}
By assumption \eqref{beta growth} and the Mean-Value Theorem again, we estimate
\begin{align*}
\left\vert \frac{\beta(u_{n+1})-\beta(u_{n})}{\tau}\right\vert  &  =\frac
{1}{\tau}|\beta^{\prime}((1-\theta_{n})u_{n+1}+\theta_{n}u_{n})|\cdot
|u_{n+1}-u_{n}|\\
&  \leq C_{\delta}\left\vert \frac{u_{n+1}-u_{n}}{\tau}\right\vert
\end{align*}
for some $\theta_{n}(x)\in(0,1)$, for a.e.~$x\in\Omega$. Thus, as a
consequence of (\ref{lp integral est}),
\[
\sum_{n=0}^{N-1}\tau\left\Vert \frac{\beta(u_{n+1})-\beta(u_{n})}{\tau
}\right\Vert _{L^{p}(\Omega)}^{p}\leq C_{\mu}\text{, }%
\]
which with $p \geq 2$ gives
\begin{equation}
\sum_{n=0}^{N-1}\tau\left\Vert \frac{\beta(u_{n+1})-\beta(u_{n})}{\tau
}\right\Vert _{L^{2}(\Omega)}^{2}\leq C_{\mu}\text{.} \label{beta est}%
\end{equation}
In case $1 < p < 2$, \eqref{beta est} follows from \eqref{lp integral est} with $p = 2$.

We further derive uniform estimates for $\xi_{n+1}$ and $\Delta u_{n+1}$.

  \begin{lemma}[Estimates for $\xi_n$]\label{xi regularity}
   It is satisfied that
   \begin{equation}\label{e:xi:2}
       \|\xi_{n+1}\|_{L^2(\Omega)} \leq \|(\LAM\Delta u^0+\gamma(u^0))_-\|_{L^2(\Omega)}.
   \end{equation}
   Moreover, there exists a constant $C_{\mu}\geq0$ depending on $\mu$ and $u^{0}$ such that%
\begin{equation}
\sum_{n=0}^{N-1}\tau\left\Vert \Delta u_{n+1}\right\Vert _{L^{p^{\prime}%
}(\Omega)}^{p^{\prime}}\leq C_{\mu}. \label{H2 bound}%
\end{equation}
   In addition, assume that $(\LAM\Delta u^{0}+\gamma(u^{0}))_{-}\in L^{q}(\Omega)$ for some $q\in (2+\infty]$. Then it holds that
\[
\left\Vert \xi_{n+1}\right\Vert _{L^{q}(\Omega)}\leq\left\Vert (\LAM\Delta u^{0}+\gamma(u^{0}))_{-}\right\Vert _{L^{q}(\Omega)}.
\]
 \end{lemma}

\begin{proof}
In this proof, we shall establish a second energy estimate by differentiating (in time) the discretized equation. To this end, we start with generating an additional data of $u_n$ (for $n = -1$). Set
$$
\alpha(x,z) := \mu \gamma\left( \dfrac{u_0(x)-z}\tau\right) + \dfrac{\beta(u_0(x))-\beta(z)}\tau
$$
for $z \in \R$ and $x \in \Omega$. Then $\alpha(x,z)$ is continuous and strictly decreasing in $z$ for a.e.~$x \in \Omega$ and it holds by assumptions that
$$
\lim_{z \to \pm \infty} \alpha(x,z) = \mp \infty,
$$
that is, the range of $\alpha(x,\cdot)$ coincides with $\R$. Hence there exists a measurable function $z(x)$ such that
$$
\alpha(x,z(x)) = \Delta u_0(x) + \gamma(u_0(x)) \quad \mbox{ for a.e. } x \in \Omega.
$$
Then one can check that%Here we claim that
$$
z(x) < u_0(x) \quad \mbox{ if } \ \Delta u_0(x) + \gamma(u_0(x)) > 0.
$$
%Indeed, suppose on the contrary that $z(x) \geq u_0(x)$. Then by the relation above,
%$$
%0 \geq \alpha(x,z(x)) = \Delta u_0(x) + \gamma(u_0(x)),
%$$
 %which is a contradiction. Therefore $z(x) < u_0(x)$.
 Now, set
$$
u_{-1}(x) := \begin{cases}
	      u_0(x) &\mbox{if } \ \Delta u_0(x)+\gamma(u_0(x)) \leq 0,\\
	      z(x) &\mbox{if } \ \Delta u_0(x)+\gamma(u_0(x)) > 0.
	     \end{cases}
	     $$
	     Then it follows that $u_0 \geq u_{-1}$, and moreover,
\begin{align}
 \mu \gamma \left( \dfrac{u_0 - u_{-1}}\tau \right) + \dfrac{\beta(u_0)-\beta(u_{-1})}\tau + \xi_0
 &= \Delta u_0 + \gamma(u_0),\label{s1}\\
 \xi_0 &\in \partial I_{[0,+\infty)}\left( \dfrac{u_0-u_{-1}}\tau \right),\label{s2}
\end{align}
 which corresponds to \eqref{discrete euler} and \eqref{discrete euler2} with $n = -1$.

Let $R>0$ and define $\eta_{n+1}=G_{R}(\xi_{n+1})\in L^{\infty}(\Omega)$,
where $G_{R}\in C^{1}(\mathbb{R} )$ is a monotone function satisfying%
\[
G_{R}(u)=\begin{cases}
	  |u|^{q-2}u & \text{if } \ |u|\leq R\text{,}\\
	  \mathrm{sign}(u)(R+1)^{q-1} & \text{if } \ |u|\geq R+2 \text{.}%
	 \end{cases}
\]
By subtraction of equations and test by $\eta_{n+1}$, we get
\begin{align}
&  \mu\left(  \gamma\left(  \frac{u_{n+1}-u_{n}}{\tau}\right)  -\gamma\left(
\frac{u_{n}-u_{n-1}}{\tau}\right)  ,\eta_{n+1}\right)  _{L^{p}(\Omega
)}\nonumber\\
&\quad  +\left(  \frac{\beta\left(  u_{n+1}\right)  -\beta\left(  u_{n}\right)
}{\tau}-\frac{\beta(u_{n})-\beta(u_{n-1})}{\tau},\eta_{n+1}\right)
_{L^{2}(\Omega)}\nonumber\\
&\quad  +\left(  \xi_{n+1}-\xi_{n},\eta_{n+1}\right)  _{L^{p}\left(  \Omega\right)
}+(-\LAM\Delta(u_{n+1}-u_{n}),\eta_{n+1})_{L^{p}(\Omega)}\nonumber\\
&  =\left(  \gamma(u_{n+1})-\gamma(u_{n}),\eta_{n+1}\right)  _{L^{p}\left(
\Omega\right)  } \label{t eq}%
\end{align}
 for $n = 0,1,\cdots,N-1$. Note that $(-\LAM\Delta(u_{n+1}-u_{n}),\eta_{n+1})_{L^{p}(\Omega)} \geq 0$ by integration by parts.
 Here we also used the fact that $u_0 = u^0 \in D(\Delta)$. Indeed, as a consequence of \cite[Prop. A.2]{Ak},
we have%
\[
\int_{\Omega}-\Delta u\eta\geq0 \quad \text{ for all } \ \eta\in\partial I_{[0,+\infty
)}(u)\text{ and }u\in D(-\Delta)\text{ satisfying } u\geq0\text{. }%
\]
Moreover, recalling the non-decrease $u_{n+1}\geq u_{n}$ a.e.~in $\Omega$, the positivity of $\gamma$ by (A1), and the strict monotonicity of $\beta$ by (A3), we observe that
\begin{align}
\partial I_{[0,+\infty)}\left(  \frac{\beta(u_{n+1})-\beta(u_{n})}{\tau
}\right)   &  =\partial I_{[0,+\infty)}\left(  \frac{u_{n+1}-u_{n}}{\tau
}\right) \nonumber\\
&  =\partial I_{[0,+\infty)}\left(  \gamma\left(  \frac{u_{n+1}-u_{n}}{\tau
}\right)  \right)  \text{ a.e.~in }\Omega\text{.}\label{dI}%
\end{align}
Then, by using the definition of subdifferential, as $\eta_{n+1}%
\in\partial I_{[0,+\infty)}\left(  \gamma\left(  \frac{u_{n+1}-u_{n}}{\tau
}\right)  \right)$, one has
\begin{align*}
&  \int_{\Omega}\left(  \gamma\left(  \frac{u_{n+1}-u_{n}}{\tau}\right)
-\gamma\left(  \frac{u_{n}-u_{n-1}}{\tau}\right)  \right)  \eta_{n+1}\\
&  \geq I\left(  \gamma\left(  \frac{u_{n+1}-u_{n}}{\tau
}\right)  \right) -I\left(  \gamma\left(  \frac{u_{n}-u_{n-1}%
}{\tau}\right)  \right)  =0
\end{align*}
and similarly by \eqref{dI},%
\[
\int_{\Omega}\left(  \frac{\beta(u_{n+1})-\beta(u_{n})}{\tau}-\frac
{\beta\left(  u_{n}\right)  -\beta(u_{n-1})}{\tau}\right)  \eta_{n+1}%
\geq0\text{.}%
\]
Note that $\eta_{n+1}(x)\neq0$ only if $u_{n+1}(x)=u_{n}(x)$. It follows that
\[
\int_{\Omega}\left(  \gamma\left(  u_{n+1}\right)  -\gamma\left(
u_{n}\right)  \right)  \eta_{n+1}=0.
\]
Combining the above estimates and (\ref{t eq}), one gets
\[
\left(  \xi_{n+1}-\xi_{n},\eta_{n+1}\right)  _{L^{p}\left(  \Omega\right)
}\leq0
\]
 for $n = 0,1,\ldots,N-1$. By using the monotonicity of $G_{R}$ and the definition of subdifferential, we
deduce
\[
\int_{\Omega}\hat{G}_{R}(\xi_{n+1})\leq\int_{\Omega}\hat{G}_{R}(\xi_{n})
\]
for all $n=0,1,...,N-1$, where $\hat{G}_{R}$ is the primitive function of
$G_{R}$ such that $\hat{G}_{R}(0)=0$. Passing to the limit as $R\rightarrow
+\infty$, we obtain%
\begin{equation}
\int_{\Omega}|\xi_{n+1}|^{q}\leq\int_{\Omega}|\xi_{n}|^{q} \label{est}%
\end{equation}
for $n = 0,1,\ldots,N-1$.

We claim that
\begin{equation}
0\leq|\xi_{n}|=-\xi_{n} = (\LAM\Delta u_{n}+\gamma(u_{n}%
))_{-} \ \mbox{ a.e.~in } \Omega \ \mbox{ for } \ n = 0,1,\ldots,N. \label{estimate xi}%
\end{equation}
Indeed, we recall again that $\xi_{n}(x)\neq0$ only if $u_{n}(x)=u_{n-1}(x)$. By (\ref{discrete euler}), \eqref{discrete euler2}, \eqref{s1}, \eqref{s2} and $\gamma(0)=0$, we deduce that either $\xi_{n}(x)=0$ (then $\LAM\Delta u_{n}(x)+\gamma(u_{n}(x)) \geq 0$) or
\[
0> \xi_{n}(x)=\LAM\Delta u_{n}(x)+\gamma(u_{n}(x))
\]
holds for $n = 0,1,2,\ldots, N$. Thus, (\ref{estimate xi}) holds true.

By comparison in relation (\ref{est}), we have
\[
\left\Vert \xi_{n+1}\right\Vert _{L^{q}(\Omega)}\leq\left\Vert \xi
 _{n}\right\Vert _{L^{q}(\Omega)}\leq \cdots \leq\left\Vert \xi_{1}\right\Vert_{L^{q}(\Omega)}
 \leq\left\Vert \xi_{0}\right\Vert_{L^{q}(\Omega)}\leq\left\Vert (\LAM\Delta u^{0}+\gamma(u^{0}%
))_{-}\right\Vert _{L^{q}(\Omega)}.
\]
In the case $(\LAM\Delta u^{0}+\gamma(u^{0}))_{-}\in L^{\infty}%
(\Omega)$, we can pass to the limit as $q\rightarrow+\infty$ in both sides and
conclude
\begin{equation}
\left\Vert \xi_{n}\right\Vert _{L^{\infty}(\Omega)}\leq\left\Vert (\LAM\Delta u^{0}+\gamma(u^{0}))_{-}\right\Vert _{L^{\infty}(\Omega)}\text{.}
\label{xi est}%
\end{equation}
Finally, estimate (\ref{H2 bound}) follows by comparison in equation
(\ref{discrete euler}) and by using (\ref{gamma est}), (\ref{gamma est2}), and
(\ref{beta est}).
 \end{proof}

\subsection{Passage to the limit.}

We introduce the piecewise constant interpolants $\bar{u}_{\tau}, \bar{\xi}_{\tau}$ and piecewise affine interpolants $u_{\tau}, v_\tau$ defined by
\begin{align*}
\bar{u}_{\tau}(t) &:=u_{n+1}, \quad \bar{\xi}_{\tau}(t):=\xi_{n+1},\\
u_{\tau}(t)  &  :=\frac{t_{n+1}-t}{\tau}u_{n}+\frac{t-t_{n}}{\tau}
u_{n+1}\text{,}\\
v_{\tau}(t)  &  :=\frac{t_{n+1}-t}{\tau}\beta(u_{n})+\frac{t-t_{n}}{\tau}%
\beta(u_{n+1}) \quad \mbox{ for } \ t\in\lbrack t_{n},t_{n+1}),%
\end{align*}
for $n\in\{0,...,N-1\}$. Then, system
(\ref{discrete euler})-(\ref{discrete euler2}) can be rewritten as
\begin{align}
\mu\gamma(\partial_{t}u_{\tau})+\partial_{t}v_{\tau}+\bar{\xi}_{\tau}  &
=\LAM\Delta\bar{u}_{\tau}+\gamma(\bar{u}_{\tau})\text{,}%
\label{discrete euler 2}\\
 \bar{\xi}_{\tau}  &  \in\partial I_{[0,+\infty)}(\partial_{t}u_{\tau})\text{.}%
 \nonumber
\end{align}
Thanks to the a-priori estimates above, we can extract a (not
relabeled) subsequence such that the following convergences hold:%
\begin{alignat}{4}
\bar{u}_{\tau}  &  \rightarrow\bar{u} \quad &&\text{ weakly * in }L^{\infty}%
 (0,T;H^{1}(\Omega)),%\label{conv u bar lp}
 \nonumber\\
u_{\tau}  &  \rightarrow u \quad &&\text{ weakly * in }L^{\infty}(0,T;H^{1}(\Omega)),\nonumber\\
\partial_{t}u_{\tau}  &  \rightarrow\partial_{t}u \quad &&\text{ weakly in }%
L^{p}(0,T;L^{p}(\Omega)),\label{conv dt ubar lp}\\
v_{\tau}  &  \rightarrow v \quad &&\text{ weakly * in }L^{\infty}(0,T;L^{p}(\Omega)),\nonumber\\
\partial_{t}v_{\tau}  &  \rightarrow\partial_{t}v \quad && \text{ weakly in }%
L^{p}(0,T;L^{p}(\Omega)),\nonumber\\
\beta(\bar{u}_{\tau})  &  \rightarrow\bar{v} \quad &&\text{ weakly * in }L^{\infty
 }(0,T;L^{p}(\Omega)),\nonumber\\
 \gamma(\partial_{t}u_{\tau})  &  \rightarrow\gamma \quad &&\text{ weakly in
}L^{p^{\prime}}(0,T;L^{p^{\prime}}(\Omega)),\nonumber\\
\bar{\xi}_{\tau}  &  \rightarrow\xi \quad &&\text{ weakly * in } L^{\infty}(0,T; L^2(\Omega)),\label{xi conv}\\
\Delta\bar{u}_{\tau}  &  \rightarrow\Delta\bar{u} \quad &&\text{ weakly in
}L^{p^{\prime}}(0,T;L^{p^{\prime}}(\Omega)),\nonumber
\end{alignat}
for some limits
\begin{align*}
u  &  \in W^{1,p}(0,T;L^{p}(\Omega))\cap L^{\infty}(0,T;H^{1}(\Omega
))\text{,}\\
\bar{u}  &  \in L^{\infty}(0,T;H^{1}(\Omega))\cap L^{p^{\prime}}%
(0,T;W^{2,p^{\prime}}(\Omega))\text{,}\\
 v  &  \in W^{1,p}(0,T;L^{p}(\Omega)),%\cap L^{\infty}(0,T;L^{p}(\Omega))\text{,}\\
 \quad \bar{v} \in L^{\infty}(0,T;L^{p}(\Omega))\text{,} \\
 \xi &\in L^{\infty}(0,T; L^2(\Omega)), \quad
 \bar{\gamma} \in L^{p^{\prime}}(0,T;L^{p^{\prime}}(\Omega))\text{.}%
\end{align*}
Furthermore, from Ascoli's Compactness Lemma (see, e.g., \cite{Si}) along with
estimate (\ref{h1 est}) and \ the compact embedding $H^{1}(\Omega
)\hookrightarrow L^{r}\left(  \Omega\right)  $ for all $1\leq r<2^{\ast}$, it
follows that
\[
u_{\tau}\rightarrow u\text{ strongly in }C\left(  [0,T];L^{r}\left(
\Omega\right)  \right)  \text{ for all }r \in [1,2^{\ast})\text{.}%
\]
Observe that, thanks to estimate (\ref{lp integral est}), recalling that $p>1$, we
have
\[
||u_{\tau}(t)-\bar{u}_{\tau}(t)||_{L^{p}\left(  \Omega\right)  }^{p}=\left(
\frac{t_{n+1}-t}{\tau}\right)  ^{p}||u_{n+1}-u_{n}||_{L^{p}\left(
\Omega\right)  }^{p}\leq C_\mu \tau^{p-1}\rightarrow0,
\]
which yields $u=\bar{u}$ and
\[
\bar{u}_{\tau}\rightarrow u\text{ strongly in }L^{\infty}(0,T;L^{r}%
(\Omega))\text{ for all }r\in [1,2^{\ast}).
\]
Indeed, we can derive the convergence above for $r \in [1,2^*) \cap [1,p]$ and then remove the restriction on $[1,p]$ by \eqref{h1 est}.
One can similarly verify $v=\bar{v}$. In particular, as a consequence of the continuity of $\gamma$ and of
assumption (\ref{gamma4}), we get
\begin{equation}
\gamma(\bar{u}_{\tau})\rightarrow\gamma(u)\text{ strongly in } L^\infty(0,T;L^{p^{\prime}}(\Omega)). \label{gamma strong convergence}%
\end{equation}
Due to the demiclosedness of maximal monotone operators
(see, e.g., \cite[Prop. A.1]{Ak}), we identify $v=\beta(u)$.

Now we are ready to pass to the limit in equation
(\ref{discrete euler 2}) and obtain
\begin{align}
\mu\bar{\gamma}+\partial_{t}\beta(u)+\xi &  =\LAM\Delta u+\gamma\left(
u\right)  \text{,}\label{eq with gamma}\\
u(0)  &  =u^{0}\text{.}\nonumber%
\end{align}
Note also that $\partial_{t}u\geq0$ a.e.~in $\Omega\times(0,T)$, since
$\partial_{t}u_{\tau}$ is non-negative. We now identify the limit $\xi$ as a
section of $\partial I_{[0,+\infty)}(\partial_{t}u)$. By $p\geq2$, note that
$\partial_{t}u_{\tau}$ and $\partial_{t}v_{\tau}\ $are bounded in
$L^{2}(0,T;L^{2}(\Omega)).$ By comparison in equation (\ref{discrete euler 2}%
), we have
\begin{align*}
&  \limsup_{\tau\rightarrow0}\int_{0}^{T}(\bar{\xi}_{\tau},\partial_{t}%
u_{\tau})_{L^{2}(\Omega)}\\
&  =\limsup_{\tau\rightarrow0}\left\{  \int_{0}^{T}\left(  -\partial
_{t}v_{\tau},\partial_{t}u_{\tau}\right)  _{L^{2}\left(  \Omega\right)  }%
+\int_{0}^{T}\left(  \LAM\Delta\bar{u}_{\tau}-\mu\gamma(\partial
_{t}u_{\tau})+\gamma(\bar{u}_{\tau}),\partial_{t}u_{\tau}\right)
_{L^{p}\left(  \Omega\right)  }\right\} \\
&  \leq-\frac{\LAMuno}{2}\liminf_{\tau\rightarrow0}||\nabla u_{\tau
}(T)||_{L^{2}(\Omega)}^{2}+\frac{\LAMuno}{2}||\nabla u^{0}||_{L^{2}%
(\Omega)}^{2}-\liminf_{\tau\rightarrow0}\int_{0}^{T}\left(  \partial
_{t}v_{\tau},\partial_{t}u_{\tau}\right)  _{L^{2}\left(  \Omega\right)  }\\
&  \quad -\liminf_{\tau\rightarrow0}\int_{0}^{T} \left( \mu\gamma(\partial
 _{t}u_{\tau}),\partial_{t}u_{\tau}\right)_{L^{p}(\Omega)}
 +\lim_{\tau\rightarrow0}\int_{0}^{T}\left(  \gamma(\bar{u}_{\tau}%
),\partial_{t}u_{\tau}\right)  _{L^{p}\left(  \Omega\right)  }\text{.}%
\end{align*}
Thanks to lower semicontinuity of the norm and convergence $\bar{u}_{\tau
}(T)\rightarrow u(T)$ weakly in $H^{1}(\Omega)$, we have
\[
\frac{\LAMuno}{2}\liminf_{\tau\rightarrow0}||\nabla u_{\tau}%
(T)||_{L^{2}(\Omega)}^{2}\geq \frac{\LAMuno}{2}||\nabla u(T)||_{L^{2}%
(\Omega)}^{2}.
\]
Arguing as in \cite[Lemma 3.7]{Ak}, we can prove that%
\[
\liminf_{\tau\rightarrow0}\int_{0}^{T}\left(  \partial_{t}v_{\tau}%
,\partial_{t}u_{\tau}\right)  _{L^{2}\left(  \Omega\right)  }\geq \int_{0}%
^{T}\left(  \partial_{t}v,\partial_{t}u\right)  _{L^{2}\left(  \Omega\right)
}\text{.}%
\]
Note that, as a consequence\ of strong convergence
(\ref{gamma strong convergence}), we have%
\[
\lim_{\tau\rightarrow0}\int_{0}^{T}\left(  \gamma(\bar{u}_{\tau}),\partial
_{t}u_{\tau}\right)  _{L^{p}\left(  \Omega\right)  }=\int_{0}^{T}\left(
\gamma(u),\partial_{t}u\right)  _{L^{p}\left(  \Omega\right)  }\text{.}%
\]
Finally, as a consequence of the monotonicity of $\gamma$ and of convergence
(\ref{conv dt ubar lp}), we have
\[
\liminf_{\tau\rightarrow0}\int_{0}^{T}\left(  \mu\gamma(\partial_{t}u_{\tau
}),\partial_{t}u_{\tau}\right)  _{L^{p}\left(  \Omega\right)  }\geq\int
_{0}^{T}\left(  \mu\bar{\gamma},\partial_{t}u\right)  _{L^{p}\left(
\Omega\right)  }.
\]
Thus, using (\ref{eq with gamma}), we estimate%
\begin{align*}
\limsup_{\tau\rightarrow0}\int_{0}^{T}(\bar{\xi}_{\tau},\partial_{t}u_{\tau
})_{L^{2}(\Omega)}  &  \leq-\frac{\LAMuno}{2}||\nabla u(T)||_{L^{2}%
(\Omega)}^{2}+\frac{\LAMuno}{2}||\nabla u^{0}||_{L^{2}(\Omega)}^{2}%
-\int_{0}^{T}\left(  \partial_{t}v,\partial_{t}u\right)  _{L^{2}\left(
\Omega\right)  }\\
& \quad -\int_{0}^{T}\left(  \mu\bar{\gamma},\partial_{t}u\right)  _{L^{p}\left(
\Omega\right)  }+\int_{0}^{T}\left(  \gamma(u),\partial_{t}u\right)
_{L^{p}\left(  \Omega\right)  }\\
&  =\int_{0}^{T}(\xi,\partial_{t}u)_{L^{2}(\Omega)}\text{.}%
\end{align*}
By using the maximal monotonicity of $\partial_{L^{ 2}(\Omega)}I$, we
have $\xi\in\partial_{L^{ 2}\left(  \Omega\right)  }I(\partial_{t}u)$
and hence $\xi\in\partial I_{[0,+\infty)}(\partial_{t}u)$ a.e.~in
$\Omega\times (0,T)$. Moreover, we have (see, e.g., \cite[Prop. A.1]{Ak})%
\[
\lim_{\tau\rightarrow0}\int_{0}^{T}(\bar{\xi}_{\tau},\partial_{t}u_{\tau
})_{L^{2}(\Omega)}=\int_{0}^{T}(\xi,\partial_{t}u)_{L^{2}(\Omega)}\text{.}%
\]

We next identify $\bar{\gamma}=\gamma(\partial_{t}u)$. To this aim, we estimate
\begin{align*}
&  \limsup_{\tau\rightarrow0}\int_{0}^{T}\left(  \mu\gamma(\partial_{t}%
u_{\tau}),\partial_{t}u_{\tau}\right)  _{L^{p}\left(  \Omega\right)  }\\
&  \leq\limsup_{\tau\rightarrow0} \int_{0}^{T}\left(  -\partial
_{t}v_{\tau},\partial_{t}u_{\tau}\right)  _{L^{2}(\Omega)} \\
&  \quad +\limsup_{\tau\rightarrow0}\int_{0}^{T}\left(  \LAM\Delta\bar
{u}_{\tau},\partial_{t}u_{\tau}\right)  _{L^{p}(\Omega)}+\lim_{\tau
 \rightarrow0}\int_{0}^{T} \left[(-\bar{\xi}_{\tau},\partial_{t}u_{\tau})_{L^{2}(\Omega)}
+ (\gamma(\bar{u}_{\tau}),\partial_{t}u_{\tau})_{L^{p}(\Omega)} \right].
\end{align*}
Arguing as above, we get
\[
\limsup_{\tau\rightarrow0}\int_{0}^{T}\left(  \mu\gamma(\partial_{t}u_{\tau
}),\partial_{t}u_{\tau}\right)  _{L^{p}\left(  \Omega\right)  }\leq\int
_{0}^{T}\left(  \mu\bar{\gamma},\partial_{t}u\right)  _{L^{p}\left(
\Omega\right)  }.
\]
By the demiclosedness of maximal monotone operators, it follows that
$\bar{\gamma}=\gamma(\partial_{t}u)$.

\bigskip

\begin{remark}[Proof for the case $1 < p < 2$]
{\rm
 \label{p small} A similar conclusion to Proposition \ref{P:aprx} for $p \geq 2$ can be obtained also for $1 < p<2$. In this case
the regularized equation reads $\partial_{t}\beta\left(  u\right)  +\xi
+\mu\partial_{t}u=\LAM\Delta u+\gamma(u)$. By testing the corresponding
discrete equation by $\partial_{t}u_{\tau}$ and by simply estimating
$\int_{\Omega}\gamma(u)\partial_{t}u_{\tau}\leq\alpha\int_{\Omega}%
|\partial_{t}u_{\tau}|^{2}+C_{\alpha}\int_{\Omega}|\gamma(u)|^{2}\leq
\alpha\int_{\Omega}|\partial_{t}u_{\tau}|^{2}+C_{\alpha}\int_{\Omega}%
|u|^{2}+C_{\alpha}$ for every $\alpha>0$ and some $C_{\alpha}$ we can obtain
the a-priori estimates (\ref{discrete unif est1}),
(\ref{lp integral est}), (\ref{h1 est}), (\ref{gamma est}), (\ref{H2 bound})},
and \emph{(\ref{xi est}) where $p$ is replaced by $2$ and, hence, analogous
convergence results which are enough to pass to the limit as $\tau\rightarrow0$.
 }
\end{remark}

\section{Variational comparison principle}\label{S:CP}

In order to pass to the limit as $\mu\rightarrow0$, we establish a uniform (in $\mu$) estimates for solutions $u_\mu$ to $(P)_{\mu}$. To this end, we  compare $u_\mu$ with a  supersolution constant-in-space and independent of $\mu$. On the other hand, we emphasize that solutions to $(P)_{\mu}$ might be non-unique, and hence, no standard comparison principle can be expected for general solutions. In~\cite{Ak}, a similar difficulty has already arisen and  has been overcome by proving a (standard) comparison  principle for supersolutions and \emph{strictly increasing} subsolutions and by  constructing a  strictly increasing subsolution with the aid of a specific structure of the equation with $\gamma(s) = s$.  Therefore in~\cite{Ak} the linearity of $B$ in the form \eqref{mDNE} was crucial and the result could not be extended to genuinely doubly nonlinear cases (cf.~Introduction). In this section, we develop a comparison principle  for \emph{variationally selected} solutions to problem $(P)_{\mu}$; more precisely, given initial data $w^{0}$, $u^{0}$, and $v^{0}$ satisfying $w^{0}\leq u^{0}\leq v^{0}$ a.e.~in $\Omega$,  we prove the existence of solutions  $w,u,v$ to $(P)_{\mu}$ satisfying $w(0)=w^{0}$, $u\left(  0\right)  =u^{0}$, $v\left( 0\right)  =v^{0}$ such that $w\leq u\leq v$ a.e.~in $\Omega\times(0,T)$ (see Proposition \ref{comp princ} below for more details). We further prove that every solution (constructed as in \S \ref{approx prob existence}) to $(P)_{\mu}$ with a constant initial datum is also constant in space (see Lemma \ref{a solution}) and (standard) comparison principle holds for some ODE solutions (see Lemma \ref{comp. princ. cost in space sol.}).
Thus, by combining these facts with the variational comparison principle above, we shall construct upper and lower bounds for such variationally selected solutions to $(P)_{\mu}$ uniformly for $\mu$ (see Proposition \ref{l infty est prop} below).

In what follows, we shall employ a notion of \emph{A-solution} defined by

\begin{definition}[A-solution]
A solution $u = u(x,t)$ to problem $(P)_{\mu}$ is called \emph{A-solution} of $(P)_{\mu}$ if it can be obtained as a limit of some solutions to the discretized problem \eqref{discrete euler}--\eqref{discrete euler 3} as in Section \emph{\ref{approx prob existence}}.
\end{definition}

We need the following

 \begin{lemma}[Constant-in-space solutions]\label{a solution}
  Assume that {\rm \textbf{(A1), (A3)}} are satisfied. Let $u^{0} > 0$ be a constant function over $\Omega$. Then, any A-solution to $(P)_{\mu}$ is constant-in-space over $\Omega$. Moreover, it solves $(P)_{\mu}$ with $\xi=0$.
\end{lemma}

\begin{proof}
We shall prove by induction that any minimizer $u_{n}$ of the functional $J_{n}$ defined by (\ref{J functional}) is constant over $\Omega$. By assumption, $u_0 = u^{0}$ is constant over $\Omega$. Assuming $u_{n}$ to be constant we claim that any minimizer $u_{n+1}$ of the functional $J_{n+1}$ is constant. Indeed, as is shown in Section \ref{approx prob existence}, the function
\[
s\in\lbrack u_{n},+\infty)\longmapsto F_{u_{n}}(s)=\mu\tau\hat{\gamma}\left(
\frac{s-u_{n}}{\tau}\right)  +\frac{1}{\tau}\hat{\beta}(s)-\hat{\gamma
}(s)-\frac{\beta(u_{n})}{\tau}s
\]
 is bounded from below on $[u_{n},+\infty)$ and it is of class $C^1$ and coercive (in $\R$) for $\tau > 0$ small enough. Hence it admits at least one minimizer. %Indeed, for each $a \geq \delta$ and $\tau > 0$ small enough, we note that the function
%\begin{equation}\label{a4-func}
% G_a(\cdot) := \mu\gamma\left(  \frac{\cdot-a}{\tau}\right)  +\frac{\beta(\cdot)-\beta(a)}{\tau}-\gamma(\cdot)
%\end{equation}
% has at least one zero on $[\delta,+\infty)$, since $G_a$ is continuous, $G_a(a) = - \gamma(a) < 0$ by \eqref{gamma-zero} and $G_a(s)$ diverges to $+\infty$ as $s \to +\infty$. The divergence of $G_a(s)$ can be observed as follows: we deduce by (A1) (and also (i) of Remark \ref{R:Assu}) and monotonicity of $\beta$ that
%\begin{align*}
%g(s) &:= G_a(s) \frac{s-a}\tau = \left[\mu \gamma\left( \dfrac{s-a}\tau \right) + \dfrac{\beta(s)-\beta(a)}{\tau} - \gamma(s) \right] \dfrac{s-a}\tau\\
%&\geq \mu C_1 \left( \left| \dfrac{s-a}\tau \right|^p - 1\right) - \left[ C_2(|s|^p+1)\right]^{1/p'} \left| \dfrac{s-a}\tau \right|.
%\end{align*}
% Hence for $\tau>0$ small enough that $\mu C_1/\tau^{p-1}$ is larger than $C_2^{1/p'}$, it follows that $g(s)/|(s-a)/\tau| \to + \infty$ as $s \to + \infty$. Thus $G_a(s)\to+\infty$ as $s \to +\infty$. Moreover, the set of zeros in $(a,+\infty)$ of $G_a(\cdot)$ lies on a bounded subinterval of $[a,+\infty)$. % whose length tends to zero as $\tau \to 0$. We next claim that $G_a$ is convex around $a$.
 %Therefore $F_{u_n}$ admits at least one minimizer.
 Noting that $F_{u_n}'(u_n) = -\gamma(u_n) < 0$ by (A1) and $u_n \geq u_0 > 0$, we find that $u_n$ never minimizes the function $F_{u_n}$. Now, let us recall that the functional $J_{n+1}$ can be decomposed as
\begin{equation*}
J_{n+1}(u)  = \int_{\Omega}F_{u_{n}}(u)+ \tau I\left(  \frac{u-u_{n}}{\tau}\right)
+\frac{\LAMuno}{2}\left\Vert \nabla u\right\Vert _{L^{2}(\Omega)}^{2}.
\end{equation*}
 %We have already noted that the set of minimizers of $F_{u_{n}}$ is finite and nonempty.
 Let $u_{n+1}$ be a minimizer of $F_{u_{n}}$ and let $u\in H^{1}(\Omega)$ be such that $u\geq u_{n}$ a.e.~in $\Omega$. If $u$ is non-constant, i.e., $\nabla u \not\equiv 0$, it then follows that
 $$
 J_{n+1}(u) > \int_{\Omega}F_{u_{n}}(u) \geq \int_{\Omega}F_{u_{n}}(u_{n+1}) = J_{n+1}(u_{n+1}).
 $$
 Thus, every minimizer of $J_{n+1}$ must be constant. In particular, $u_{n+1}$ is a minimizer of $J_{n+1}$ due to the fact that $J_{n+1}(w) \geq J_{n+1}(u_{n+1})$ for any constant function $w$. Since all the minimizers of $F_{u_n}$ are strictly greater than $u_n$, each minimizer, say $u_{n+1}$, of $J_{n+1}$ solves equation (\ref{discrete euler}) with $\xi_{n+1}=0$. By passing to the limit as $\tau=T/N\rightarrow0$, we get the conclusion of the lemma.
\end{proof}

\begin{lemma}[Comparison of ODE solutions]\label{comp. princ. cost in space sol.}Let $v^{0}>u^{0}>0$ be real numbers. Let $u,v\in C([0,+\infty)) \cap C^{1}\left( (0,+\infty) \right)$ be non-decreasing functions such that%
\begin{align}
\mu\gamma(u_{t})+\partial_{t}\beta\left(  u\right)   &  \leq\gamma
(u)\text{,}\label{ineq 1}\\
\mu\gamma(v_{t})+\partial_{t}\beta\left(  v\right)   &  \geq\gamma
(v)\text{,}\label{ineq2}\\
u(0) =u^{0}\text{,}\ \ v(0) & =v^{0}\text{.}\nonumber%
\end{align}
Then, $u(t)<v(t)$ for all $t>0$.
\end{lemma}

 \begin{proof}
  Set
  $$
  \tilde t := \sup \left\{ \tau > 0 \colon u(t) < v(t) \ \mbox{ for all } \ t \in [0,\tau] \right\} \in (0,+\infty].
  $$
  %  By continuity of $u - v$, there exists $\tilde{t}$ such that $u(t)<v(t)$ for all $t<\tilde{t}$.
  Suppose on the contrary that $\tilde{t}<+\infty$, which implies $u(\tilde {t})=v(\tilde{t})$. By subtracting inequalities (\ref{ineq 1}) and (\ref{ineq2}), we get
\[
\mu\left(  \gamma(u_{t})-\gamma(v_{t})\right)  +\beta^{\prime}(u)u_{t}%
-\beta^{\prime}(v)v_{t}\leq\gamma(u)-\gamma(v)\leq0 \ \text{ on } [0,\tilde{t}].
\]
Since $u_{t}\geq0$ and $\beta^{\prime}$ is non-increasing, we have
\[
\mu\left(  \gamma(u_{t})-\gamma(v_{t})\right)  +\beta^{\prime}(v)\left(
u_{t}-v_{t}\right)  \leq0 \ \text{ on } [0,\tilde{t}].
\]
  As $\mu>0$, $\beta^{\prime}>0$ and $\gamma$ is increasing (hence $r \mapsto \mu \gamma(r) + \beta'(v)r$ is strictly monotone), it follows that $u_{t} \leq v_{t}$ on $[0,\tilde{t}]$. Thus, we obtain $u(\tilde {t})<v(\tilde{t})$, which yields a contradiction.
 \end{proof}

We next establish a \emph{variational comparison principle}.

 \begin{proposition}[Variational comparison principle for $(P)_{\mu}$]\label{comp princ}
  Let $w^{0},u^{0},v^{0}$ satisfy \emph{\textbf{(A2)}} and be such that $0< w^{0}\leq
u^{0}\leq v^{0}$ a.e.~in $\Omega$. Then, there exist A-solutions $w,u,v$ to
$(P)_{\mu}$ corresponding to initial conditions $w(0)=w^{0}$, $u(0)=u^{0}$,
$v(0)=v^{0}$, respectively, such that $w\leq u\leq v$ a.e.~in $\Omega \times (0,T)$.
\end{proposition}

\begin{proof}
 For all $a\in H^{1}(\Omega)$, % satisfying $\beta(x)\in L^{2}(\Omega)$,
 define $J^{a}$ by
\begin{align*}
J^{a}(u)  &  =\mu\int_{\Omega}\tau\hat{\gamma}\left(  \frac{u-a}{\tau}\right)
+\frac{1}{\tau}\int_{\Omega}\hat{\beta}(u)+\tau I\left(
\frac{u-a}{\tau}\right) \\
&\quad +\frac{\LAMuno}{2}||\nabla u||_{L^{2}(\Omega)}^{2}-\int_{\Omega}%
\hat{\gamma}(u)-\left(  \frac{\beta(a)}{\tau},u\right)  _{L^{2}(\Omega
)}\text{.}%
\end{align*}
We claim that
\begin{equation}
 J^{a^{0}}(a\wedge b)+J^{b^{0}}(a\vee b)\leq J^{a^{0}}(a)+J^{b^{0}}(b)
  \quad \mbox{ for } \ a,b\in H^{1}(\Omega),
\label{comp princ cond}%
\end{equation}
 for any $a^{0},b^{0}\in H^{1}(\Omega)$ satisfying $a^0 \leq b^0$ a.e.~in $\Omega$. %$\beta(a^{0}),\beta(b^{0})\in L^{2}(\Omega)$.
 Indeed, we decompose $J^{a}$ as
\[
J^{a}(u)=A(u)+B^{a}(u)+C^{a}(u),
\]
where
\begin{align*}
A(u)  &  =\frac{1}{\tau}\int_{\Omega}\hat{\beta}(u)+\frac{\LAMuno}%
{2}\left\Vert \nabla u\right\Vert _{L^{2}(\Omega)}^{2}-\int_{\Omega}%
\hat{\gamma}(u)\text{,}\\
B^{a}(u)  &  =\mu\int_{\Omega}\tau\hat{\gamma}\left(  \frac{u-a}{\tau}\right)
+\tau I\left(  \frac{u-a}{\tau}\right)  \text{,}\\
C^{a}(u)  &  =-\left(  \frac{\beta(a)}{\tau},u\right)  _{L^{2}(\Omega)}.
\end{align*}
We shall prove that $A(u)$, $B^{a}(u)$, and $C^{a}(u)$ fulfill condition (\ref{comp princ cond}). One can assume $a \geq a^0$ and $b \geq b^0$ without any loss of generality (otherwise, \eqref{comp princ cond} holds immediately). By a simple calculation, we see
\begin{align*}
A(a\wedge b)+A(a\vee b)  &  =\int_{a\geq b}\left(  \frac{1}{\tau}\hat{\beta
}(b)+\frac{\LAMuno}{2}|\nabla b|^{2}-\hat{\gamma}(b)+\frac{1}{\tau}%
\hat{\beta}(a)+\frac{\LAMuno}{2}|\nabla a|^{2}-\hat{\gamma}(a)\right) \\
&  \quad +\int_{a<b}\left( \frac{1}{\tau}\hat{\beta}(a)+\frac{\LAMuno}%
{2}|\nabla a|^{2}-\hat{\gamma}(a)+\frac{1}{\tau}\hat{\beta}(b)+\frac
{\LAMuno}{2}|\nabla b|^{2}-\hat{\gamma}(b) \right) \\
&  =A(a)+A(b).
\end{align*}
Noting that $a \wedge b \geq a^0$ and $a \vee b \geq b^0$ a.e.~in $\Omega$, one observes
\begin{align*}
B^{a^{0}}(a\wedge b)+B^{b^{0}}(a\vee b)  &  =\mu\int_\Omega \tau\hat{\gamma
}\left(  \frac{a\wedge b-a^{0}}{\tau}\right)  +\tau I\left(
\frac{a\wedge b-a^{0}}{\tau}\right) \\
&  \quad +\mu\int_{\Omega}\tau\hat{\gamma}\left(  \frac{a\vee b-b^{0}}{\tau
}\right)  +\tau I \left(  \frac{a\vee b-b^{0}}{\tau}\right) \\
&  =\int_{a\geq b}\mu\tau\hat{\gamma}\left(  \frac{b-a^{0}}{\tau}\right)
%+\tau I_{[0,\infty)}\left(  \frac{b-a^{0}}{\tau}\right) \\
+\int_{a\geq b}\mu\tau\hat{\gamma}\left(  \frac{a-b^{0}}{\tau}\right)
 %+\tau I_{[0,\infty)}\left(  \frac{a-b^{0}}{\tau}\right)
 \\
&  \quad +\int_{a<b}\mu\tau\hat{\gamma}\left(  \frac{a-a^{0}}{\tau}\right) % +\tauI_{[0,\infty)}\left(  \frac{a-a^{0}}{\tau}\right) \\
+\int_{a<b}\mu\tau\hat{\gamma}\left(  \frac{b-b^{0}}{\tau}\right), % +\tau I_{[0,\infty)}\left(  \frac{b-b^{0}}{\tau}\right)
\end{align*}
and moreover, by $a \geq a^0$ and $b \geq b^0$ a.e.~in $\Omega$,
\begin{align*}
B^{a^{0}}(a)+B^{b^{0}}(b)  &  =\int_{a\geq b}\mu\tau\hat{\gamma}\left(
\frac{a-a^{0}}{\tau}\right)%  +\tau I_{[0,\infty)}\left(  \frac{a-a^{0}}{\tau}\right) \\
+\int_{a\geq b}\mu\tau\hat{\gamma}\left(  \frac{b-b^{0}}{\tau}\right)\\
%+\tau I_{[0,\infty)}\left(  \frac{b-b^{0}}{\tau}\right) \\
&\quad +\int_{a<b}\mu\tau\hat{\gamma}\left(  \frac{a-a^{0}}{\tau}\right)
% +\tau I_{[0,\infty)}\left(  \frac{a-a^{0}}{\tau}\right) \\
 +\int_{a<b}\mu\tau\hat{\gamma}\left(  \frac{b-b^{0}}{\tau}\right).
%+\tau I_{[0,\infty)}\left(  \frac{b-b^{0}}{\tau}\right)  \text{.}%
\end{align*}
Hence, in order to check (\ref{comp princ cond}) for $B^a$, it suffices to show
\begin{align*}
  \int_{a\geq b}\left\{  \mu\tau\hat{\gamma}\left(  \frac{b-a^{0}}{\tau
 }\right) % +\tau I_{[0,\infty)}\left(  \frac{b-a^{0}}{\tau}\right)
 +\mu \tau\hat{\gamma}\left(  \frac{a-b^{0}}{\tau}\right) % +\tau I_{[0,\infty)}\left(  \frac{a-b^{0}}{\tau}\right)
 \right\}
  \leq\int_{a\geq b}\left\{  \mu\tau\hat{\gamma}\left(  \frac{a-a^{0}}{\tau
 }\right) % +\tau I_{[0,\infty)}\left(  \frac{a-a^{0}}{\tau}\right)
 +\mu \tau\hat{\gamma}\left(  \frac{b-b^{0}}{\tau}\right)
 %+\tau I_{[0,\infty)}\left(  \frac{b-b^{0}}{\tau}\right)
 \right\}.
\end{align*}
 Recalling $a^{0} \leq b^{0} \leq b \leq a$ over the region of integration above, we infer that the inequality above holds true by convexity of $\hat{\gamma}$. %Indeed, one finds that
%\begin{align*}
% \hat\gamma\left(\dfrac{b-a^0}\tau\right)+\hat\gamma\left(\dfrac{a-b^0}\tau\right)
% \leq \hat\gamma\left(\dfrac{b-a^0}\tau\right)+\hat\gamma\left(\dfrac{b-a^0}\tau\right)
%\end{align*}
 Finally, we note that
\[
C^{a^{0}}(a\wedge b)+C^{b^{0}}(a\vee b)=-\int_{a\geq b}\frac{\beta(a^{0}%
)}{\tau}b-\int_{a<b}\frac{\beta(a^{0})}{\tau}a-\int_{a\geq b}\frac{\beta
(b^{0})}{\tau}a-\int_{a<b}\frac{\beta(b^{0})}{\tau}b
\]
and
\[
C^{a^{0}}(a)+C^{b^{0}}(b)=-\int_{a\geq b}\frac{\beta(a^{0})}{\tau}a-\int
_{a<b}\frac{\beta(a^{0})}{\tau}a-\int_{a\geq b}\frac{\beta(b^{0})}{\tau}%
b-\int_{a<b}\frac{\beta(b^{0})}{\tau}b\text{.}%
\]
Thus, it remains to check that
\begin{equation}
\int_{a\geq b}\frac{\beta(a^{0})}{\tau}b+\frac{\beta(b^{0})}{\tau}a\geq
\int_{a\geq b}\frac{\beta(a^{0})}{\tau}a+\frac{\beta(b^{0})}{\tau}b\text{.}
\label{eses}
\end{equation}
As $\beta$ is increasing, we have $\beta(a^{0})\leq\beta(b^{0})$. Thus, for $a\geq b$, it holds that $0\leq(a-b)(\beta(b^{0})-\beta(a^{0}))$, i.e., $\beta(a^{0})b+\beta(b^{0})a\geq\beta(a^{0})a+\beta(b^{0})b$, which yields estimate
(\ref{eses}). This proves relation (\ref{comp princ cond}).

 We are now in position to prove the lemma. Suppose that $\delta \leq w^0 \leq u^0 \leq v^0$ a.e.~in $\Omega$.
 %We now prove that there exist $w_{1}$, $u_{1}$, and $v_{1}$ minimizers of $J^{w^{0}}$, $J^{u^{0}}$, and $J^{v^{0}}$, respectively, such that $w_{1}\leq u_{1}\leq v_{1}$.
 Let $w,u,v\in H^{1}(\Omega)$ be minimizers of $J^{w^{0}}$, $J^{u^{0}}$, and $J^{v^{0}}$, respectively. Then, $w\vee v,w\wedge v\in H^{1}(\Omega)$. Moreover, by minimality we particularly have
$$
J^{w^{0}}(w) \leq J^{w^{0}}(w\wedge v),\quad J^{v^{0}}(v) \leq J^{v^{0}}(w\vee v).
$$
By using the fact (\ref{comp princ cond}) with $a=w$, $a^{0}=w^{0}$, $b=v$,
and $b^{0}=v^{0}$, we get
\[
J^{w^{0}}(w)\leq J^{w^{0}}(w\wedge v)\leq J^{w^{0}}(w)+J^{v^{0}}(v)-J^{v^{0}%
}(w\vee v)\leq J^{w^{0}}(w)\text{,}%
\]
whence follows
$$
J^{w^{0}}(w\wedge v) =J^{w^{0}}(w), \quad J^{v^{0}}(w\vee v) =J^{v^{0}}(v).
$$
Thus, $\tilde{w}:=w\wedge v$ and $\tilde{v}:=w\vee v$ also minimize $J^{w^{0}}$ and $J^{v^{0}}$, respectively, and $\tilde{w}\leq\tilde{v}$. By using (\ref{comp princ cond}) again with the choice $a=u$, $a^{0}=u^{0}$, $b=\tilde{v}$, $b^{0}=v^{0}$ and by arguing as above, we deduce that
$v_{1}:=u\vee\tilde{v}$ and $u_{1}:=u\wedge\tilde{v}$ minimize $J^{v^{0}}$ and $J^{u^{0}}$, respectively. Furthermore, (\ref{comp princ cond}) with
$a=\tilde{w}$, $a^{0}=w^{0}$, $b=u$, $b^{0}=u^{0}$ leads us to infer that $w_{1}:=\tilde{w}  \wedge u$ minimizes $J^{w^{0}}$. Obviously, recalling the relation $\tilde{w} \leq\tilde{v}$, we have $w_{1}\leq u_{1}\leq v_{1}$ a.e.~in $\Omega$.

By iterating the above argument we can construct sequences $\left\{  w_{n}\right\}$, $\left\{  u_{n}\right\}$ and $\left\{  v_{n}\right\}$ such that $w_{n+1}$, $u_{n+1}$, and $v_{n+1}$ minimize $J^{w_{n}}$, $J^{u_{n}}$, and
$J^{v_{n}}$, respectively, and $w_{n+1}\leq u_{n+1}\leq v_{n+1}$ for all
$n\in\{0,\dots,N-1\}$. By passing to the limit as $N\rightarrow\infty$ we complete the proof of the lemma.
\end{proof}

Combining all these facts, we obtain

\begin{proposition}[Uniform bounds]
\label{l infty est prop}Let $0<\delta\leq u^{0}\leq M$. Let $z_{M}\in C([0,T_{M})) \cap C^{1}(0,T_{M})$ be a solution of the ODE
\begin{align*}
\partial_{t}\beta(z_{M}(t)) =\gamma(z_{M}(t))\text{,}  \quad 0 < t < T_{M}, \quad z_{M}(0)=M,
\end{align*}
where $T_{M}>0$ is the maximal existence time for $z_{M}$, that is,
$$
T_{M} := \int^\infty_{\beta(M)} \dfrac{\d b}{\gamma \circ \beta^{-1}(b)}.
$$
 Let $z_{\delta}^{\mu}\in C([0,+\infty)) \cap C^{1}(0,+\infty)$ be a solution to
 \begin{equation}\label{zmu}
\mu\gamma(\partial_{t}z_{\delta}^{\mu}(t))+\partial_{t}\beta(z_{\delta}^{\mu}(t))
 =\gamma(z_{\delta}^{\mu}(t))\text{,} \quad 0 < t < +\infty,\quad z_{\delta}^{\mu}(0) =\delta\text{.}%
 \end{equation}
 Then, there exists an A-solution $u$ to $(P)_{\mu}$ with $u(0)=u^{0}$ such that
\[
z_{\delta}^{\mu}(t)\leq u(t,x)\leq z_{M}(t)
\]
for a.e.~$(x,t)\in\Omega\times(0,T_{M})$.
\end{proposition}

\begin{proof}
Thanks to Proposition \ref{comp princ}, there exist A-solutions $w,u,v$ to problem
 $(P)_{\mu}$ with $w(0)=\delta$, $u(0)=u$ and $v(0)=M$ such that $w\leq u\leq v$ a.e.~in $\Omega\times(0,+\infty)$. By virtue of Lemma \ref{a solution}, one can assume that $w$
and $v$ are constant in space. The functions $z_{\delta}^{\mu}%
,z_{M}^\mu:[0,T_{0})\rightarrow\mathbb{R} _{+}$ defined by $z_{\delta}^{\mu
}(t)=w(x,t)$, $z_{M}^{\mu}(t)=v(x,t)$ solve the Cauchy problems
$$
\mu\gamma(\partial_{t}z_{M}^{\mu}(t))+\partial_{t}\beta(z_{M}^{\mu}(t)) =\gamma
(z_{M}^{\mu}(t)), \quad 0 < t < +\infty, \quad z_{M}^{\mu}(0) =M
 $$
 and
$$
\mu\gamma(\partial_{t}z_{\delta}^{\mu}(t))+\partial_{t}\beta(z_{\delta}^{\mu}(t))
=\gamma(z_{\delta}^{\mu}(t))\text{,} \quad 0 < t < +\infty, \quad z_{\delta}^{\mu}(0) =\delta\text{.}%
$$
Note that $z_{M}^{\mu}$ is strictly increasing. Let $\vep > 0$ and $z_{M+\vep}$ be given similarly to $z_{M}$ of the
statement of the proposition. Then, $z_{M+\vep}$ is also strictly increasing. Thus, as
$\gamma$ takes nonnegative values, we have%
\[
\mu\gamma(\partial_{t}z_{M+\vep}(t))+\partial_{t}\beta(z_{M+\vep}(t))\geq\gamma(z_{M+\vep}(t)), \quad 0 < t < T_{M+\vep}
\]
for all $\mu >0$. By $z_M^\mu(0) < z_{M+\vep}(0)$, applying Lemma \ref{comp. princ. cost in space sol.}, one has $z_{M}^{\mu} < z_{M+\vep}$ on $(0,T_{M+\vep})$. On the other hand, by taking a limit as $\vep \to 0_+$, one can verify that $z_{M+\vep} \to z_M$ locally uniformly on $[0,T_M)$.
 Thus,
\[
 z_{\delta}^{\mu}(t)=w(x,t)\leq u(x,t)\leq v(x,t)=z_{M}^{\mu}(t)\leq z_{M}(t) \ \text{ for a.e.~}
 (x,t) \in\Omega\times(0,T_{M})\text{.}%
\]
 This completes the proof.
\end{proof}

\section{Existence of a local solution to $(P)$}\label{S:conv}

By virtue of Proposition \ref{l infty est prop}, we now obtain uniform
estimates for some solutions to $(P)_{\mu}$ which allow us to pass to the
limit as $\mu\rightarrow0$ and complete the proof of Theorem \ref{main thm}. We focus on the
case $p\geq2$ only. As for the case $1<p<2$, analogous results can be obtained by applying minor modifications to the argument presented here (see also Remark \ref{p small}).

Let $z_{M}$ be the function defined in Proposition \ref{l infty est prop} and
let $\hat{T}_{0}:= T_{M}$ be the corresponding blow-up time (set
$T_{M}=+\infty$ if $z_{M}$ is well defined over the whole half-line $[0,+\infty)$). Then, for every $T_{0}<\hat{T}_{0}$ there exists
$C_0=C_0(T_{0},M)$ (independent of $\mu$) such that
\[
\sup_{t\in\lbrack0,T_{0}]}|z_{M}(t)|\leq C_0\text{.}%
\]
As a consequence of Proposition \ref{l infty est prop}, we have, for every $\mu$, there exists an A-solution $u_{\mu}$ to problem $(P)_{\mu}$ such
that%
\begin{equation}
\sup_{t\in\lbrack0,T_{0}]}\left\Vert u_{\mu}(t)\right\Vert _{L^{\infty}%
(\Omega)}\leq C_0. \label{l infty est}%
\end{equation}

Testing equation (\ref{approx equation}) with $\partial_{t}u_{\mu}$, we obtain
\begin{align*}
 \left(  \partial_{t}\beta\left(  u_{\mu}\right)  ,\partial_{t}u_{\mu
}\right)  _{L^{2}(\Omega)}+\left(  \xi_{\mu},\partial_{t}u_{\mu}\right)
_{L^{2}(\Omega)}+\mu\left(  \gamma(\partial_{t}u_{\mu}),\partial_{t}u_{\mu
}\right)  _{L^{p}(\Omega)}+\frac{\LAMuno}{2}\frac{\mathrm{d}}{\mathrm{d}%
t}\left\Vert \nabla u_{\mu}\right\Vert _{L^{2}(\Omega)}^{2}\\
  =\frac{\mathrm{d}}{\mathrm{d}t}\int_{\Omega}\hat{\gamma}(u_{\mu})\text{.}%
\end{align*}
By using assumption (\ref{gamma2}) and noting that $\left(  \xi_{\mu}%
,\partial_{t}u_{\mu}\right)  _{L^{2}(\Omega)}=0$, we get
\[
\mu C_1 \left\Vert \partial_{t}u_{\mu}\right\Vert _{L^{p}(\Omega)}^{p}%
+\frac{\LAMuno}{2}\frac{\mathrm{d}}{\mathrm{d}t}\left\Vert \nabla u_{\mu
}\right\Vert _{L^{2}(\Omega)}^{2}+\int_{\Omega}\beta^{\prime}(u_{\mu
})|\partial_{t}u_{\mu}|^{2}\leq\frac{\mathrm{d}}{\mathrm{d}t}\int_{\Omega}%
\hat{\gamma}(u_{\mu}) + C \text{.}%
\]
Since $\beta^{\prime}$ is non-increasing and positive, we derive by \eqref{l infty est} that
\[
\mu C_1 \left\Vert \partial_{t}u_{\mu}\right\Vert _{L^{p}(\Omega)}^{p}%
+\frac{\LAMuno}{2}\frac{\mathrm{d}}{\mathrm{d}t}\left\Vert \nabla u_{\mu
}\right\Vert _{L^{2}(\Omega)}^{2}+\beta'(C_0)\int_{\Omega}|\partial_{t}u_{\mu}|^{2}%
\leq\frac{\mathrm{d}}{\mathrm{d}t}\int_{\Omega}\hat{\gamma}(u_{\mu})\text{.}%
\]
By integrating it over $(0,t)$, $t\in\lbrack0,T_{0}]$, and by using assumption (\ref{gamma1.5}) together with \eqref{l infty est}, we obtain
\[
\mu C_1 \left\Vert \partial_{t}u_{\mu}\right\Vert _{L^{p}(0,t;L^{p}(\Omega))}%
^{p}+\frac{\LAMuno}{2}\left\Vert \nabla u_{\mu}(t)\right\Vert
_{L^{2}(\Omega)}^{2}+\beta'(C_0)\left\Vert \partial_{t}u_{\mu}\right\Vert _{L^{2}%
(0,t;L^{2}(\Omega))}^{2}\leq C,
\]
which yields
\begin{align}
\mu\left\Vert \partial_{t}u_{\mu}\right\Vert _{L^{p}(0,T_{0};L^{p}(\Omega
))}^{p}  &  \leq C\text{,}\label{unif est 1}\\
\left\Vert \partial_{t}u_{\mu}\right\Vert _{L^{2}(0,T_{0};L^{2}(\Omega))}  &
\leq C\text{,}\label{unif est 2}\\
\sup_{t\in\lbrack0,T_{0}]}\left\Vert \nabla u_{\mu}(t)\right\Vert _{L^{2}%
(\Omega)}^{2}  &  \leq C\text{.}\nonumber%
\end{align}
Thanks to (\ref{bound b}) and (\ref{l infty est}), we have
\[
\sup_{t\in\lbrack0,T_{0}]}\left\Vert \beta(u_{\mu}(t))\right\Vert _{L^{\infty
}(\Omega)}\leq C\text{.}%
\]
Note that $\partial_{t}\beta(u_{\mu})=\beta^{\prime}(u_{\mu})\partial
_{t}u_{\mu}$. Using $u_{\mu} \geq \delta$, \eqref{beta growth} and estimate (\ref{unif est 2}), we get
\begin{align*}
 \left\Vert \partial_{t}\beta(u_{\mu})\right\Vert _{L^{2}(0,T_{0};L^{2} (\Omega))}^{2} \leq C. 
\end{align*}
By virtue of assumption (\ref{gamma4}) and estimates (\ref{l infty est}) and \eqref{unif est 1}, we
have
\[
\sup_{t\in\lbrack0,T_{0}]}\left\Vert \gamma(u_{\mu}(t))\right\Vert
_{L^{\infty}(\Omega)}\leq C,\quad
 \mu \left\| \gamma(\partial_t u_\mu)\right\|_{L^{p'}(0,t;L^{p'}(\Omega))}^{p'} \leq C.
\]
Thanks to assumption {\bf (A2)}, estimate \eqref{e:xi:2} and
convergence (\ref{xi conv}), we have
\begin{equation*}
\left\Vert \xi_{\mu}\right\Vert _{ L^{\infty}(0,T_0;L^2(\Omega))}\leq\left\Vert \left(  \LAM\Delta u^{0}+\gamma(u^{0})\right)
_{-}\right\Vert _{L^{2}(\Omega)}\leq C.% \label{xi mu est}%
\end{equation*}
Finally, by comparison in equation (\ref{approx equation}), we conclude that
\[
\left\Vert \Delta u_{\mu}\right\Vert _{L^{p'}(0,T_{0};L^{p'}(\Omega))}\leq
C\text{.}%
\]

Owing to the above uniform estimates, up to some (not relabeled) subsequence
$\mu\rightarrow0$, we obtain the following:%
\begin{alignat}{4}
u_{\mu}  &  \rightarrow u \quad &&\text{ weakly * in }L^{\infty}(0,T_{0};H^{1}(\Omega))\text{,}\nonumber\\
&\quad && \text{ weakly * in }L^{\infty}(\Omega \times (0,T_0))\text{,}\nonumber\\
&\quad &&  \text{ weakly in }H^{1}(0,T_{0};L^{2}(\Omega))\text{,}\label{c:u:H1L2}\\
&\quad && \text{ strongly in }C([0,T_{0}];L^{2}(\Omega))\text{,}\nonumber\\
\xi_{\mu}  &  \rightarrow\xi \quad &&\text{ weakly * in }L^{\infty}(0,T_0;L^2(\Omega))\text{,}\nonumber\\%\label{conv xi mu}\\
\beta(u_{\mu})  &  \rightarrow v \quad &&\text{ weakly * in }L^{\infty}(\Omega \times (0,T_0))\text{,}\nonumber\\
& \quad &&  \text{ weakly in }H^{1}(0,T_{0};L^{2}(\Omega))\text{,}\nonumber\\
\Delta u_{\mu}  &  \rightarrow\Delta u \quad &&\text{ weakly in }L^{p'}(0,T_{0}%
;L^{p'}(\Omega))\text{,}\nonumber\\
%\mu\partial_{t}u_{\mu}  &  \rightarrow0 \quad &&\text{ strongly in }L^{p}(0,T_{0}%
% ;L^{p}(\Omega))\text{,}\\
 \mu \gamma(\partial_t u_\mu) &\rightarrow 0 \quad &&\text{ strongly in } L^{p'}(0,T_0;L^{p'}(\Omega)),\nonumber
\end{alignat}
for some limits
\begin{align*}
u  &  \in H^{1}(0,T_{0};L^{2}(\Omega))\cap L^{\infty}(\Omega \times (0,T_0))\cap L^{\infty}(0,T_{0};H^{1}(\Omega))\text{,}\\
v  &  \in L^{\infty}(\Omega \times (0,T_0))\cap H^{1}(0,T_{0};L^{2}(\Omega))\text{,} \quad \xi \in L^{\infty}( 0,T_0;L^2(\Omega)).
\end{align*}
Here we used the fact that $\mu^{1/p'} \gamma(\partial_t u_\mu)$ is uniformly bounded in $L^{p'}(0,T_{0};L^{p'}(\Omega))$. In addition, if $(\Delta u^0+\gamma(u^0))_- \in L^q(\Omega)$ for $q \in (2,+\infty)$ ($q = +\infty$, respectively), then by Lemma \ref{xi regularity} it follows that $\xi \in L^\infty(0,T_0;L^q(\Omega))$ ($\xi \in L^\infty(\Omega \times (0,T_0))$, respectively). Furthermore, by embedding theorem (see, e.g.,~\cite{LM}) we also have $u \in C_w([0,T_0];H^1(\Omega))$.
%By $p\geq2$, we have also
%\[
%\mu\partial_{t}u_{\mu}\rightarrow0 \quad \text{ strongly in }L^{2}(0,T_{0}%
%;L^{2}(\Omega)).
%\]
Moreover,
\begin{equation}\label{gam-conv}
\gamma(u_{\mu})\rightarrow\gamma(u) \quad \text{ strongly in }L^{m}(\Omega
\times\left(  0,T_{0}\right)  )\text{ for all }m\in\lbrack1,+\infty)\text{.}%
\end{equation}
By demiclosedness of maximal monotone operators, we deduce $v=\beta(u)$.
Furthermore, $u$ solves
\begin{equation}\label{rel}
\partial_{t}\beta(u)+\xi=\LAM\Delta u+\gamma(u)\text{.}%
\end{equation}
Hence it yields that $\Delta u \in L^2(0,T_0;L^2(\Omega))$ by comparison.
Note also that, as a consequence of the above convergences, $\partial_{t}%
u\geq0$ a.e.~in $\Omega\times(0,T)$ since $\partial_{t}u_{\mu}$ is
non-negative. We now identify the limit $\xi$ as a section of $\partial
I_{[0,+\infty)}(\partial_{t}u)$. By comparison in equation
(\ref{approx equation}) and using (\ref{gamma2}), we have
\begin{align}
&  \limsup_{\mu\rightarrow0}\int_{0}^{T_{0}}(\xi_{\mu},\partial_{t}u_{\mu
})_{L^{2}(\Omega)}\nonumber\\
& \leq \limsup_{\mu\rightarrow0}\int_{0}^{T_{0}}\left(  -\partial_{t}\beta
(u_{\mu})+\gamma(u_{\mu}),\partial_{t}u_{\mu}\right)  _{L^{2}\left(
 \Omega\right)}\nonumber\\
 &\quad +\limsup_{\mu\rightarrow0}\int_{0}^{T_{0}}\left(  \LAM\Delta u_{\mu}-\mu\gamma(\partial_{t}u_{\mu}),\partial_{t}u_{\mu}\right)
_{L^{p}\left(  \Omega\right)  }\nonumber\\
&  \leq-\frac{\LAMuno}{2}\liminf_{\mu\rightarrow0}\left\Vert \nabla
u_{\mu}(T_{0})\right\Vert _{L^{2}(\Omega)}^{2}+\frac{\LAMuno}{2}\left\Vert
 \nabla u^{0}\right\Vert _{L^{2}(\Omega)}^{2}
 -\liminf_{\mu\to0}\int_{0}^{T_{0}}\left( \partial_{t}\beta(u_{\mu}),\partial_{t}u_{\mu}\right)_{L^{2}(\Omega)}
 \nonumber\\
&\quad +\lim
_{\mu\rightarrow0}\int_{0}^{T}\left(  \gamma(u_{\mu}),\partial_{t}u_{\mu
}\right)_{L^{2}\left(  \Omega\right)  }+\lim_{\mu\rightarrow0}C\mu.
\label{minty}%
\end{align}
Thanks to \eqref{c:u:H1L2} and \eqref{gam-conv}, we have
\[
\lim_{\mu\rightarrow0}\int_{0}^{T}\left(  \gamma(u_{\mu}),\partial_{t}u_{\mu
}\right)  _{L^{2}\left(  \Omega\right)  }=\int_{0}^{T}\left(  \gamma
(u),\partial_{t}u\right)  _{L^{2}\left(  \Omega\right)  }\text{.}%
\]
By the weak lower-semicontinuity of the norm,
\[
\frac{\LAMuno}{2}\liminf_{\mu\rightarrow0}\left\Vert \nabla u_{\mu}%
(T_{0})\right\Vert _{L^{2}(\Omega)}^{2}\geq\frac{\LAMuno}{2}\left\Vert
\nabla u(T_{0})\right\Vert _{L^{2}(\Omega)}^{2}\text{.}%
\]
Arguing as in \cite[Lemma 3.7]{Ak}, we can prove that%
\[
\liminf_{\mu\rightarrow0}\int_{0}^{T_{0}}\left(  \partial_{t}\beta(u_{\mu
}),\partial_{t}u_{\mu}\right)  _{L^{2}\left(  \Omega\right)  }\geq \int
_{0}^{T_{0}}\left(  \partial_{t}\beta(u),\partial_{t}u\right)  _{L^{2}\left(
\Omega\right)  }\text{.}%
\]
By substituting these facts and \eqref{rel} into inequality (\ref{minty}), we get
\[
\limsup_{\mu\rightarrow0}\int_{0}^{T_{0}}(\xi_{\mu},\partial_{t}u_{\mu
})_{L^{2}(\Omega)}\leq\int_{0}^{T_{0}}(\xi,\partial_{t}u)_{L^{2}(\Omega)}.
\]
Thus, by the demiclosedness of maximal monotone operators, it follows that $\xi\in \partial_{L^{2}(\Omega)}I(\partial_{t}u)$, i.e., $\xi\in\partial
I_{[0,+\infty)}(\partial_{t}u)$ a.e.~in $\Omega\times\left(  0,T_{0}\right)
$. Thus, $u$ solves $(P)$ on $[0,T_0]$.

Furthermore, note that, if condition (\ref{sublinear growth}) is satisfied,
then, for every $M\in\mathbb{N}$, there exists a solution $z_{M}$ to the Cauchy
problem
\begin{align}
\partial_{t}\beta(z_{M}) =\gamma(z_{M})\text{,}\label{ode1} \quad
z_{M}(0) =M
\end{align}
over the whole half-line $[0,+\infty)$. Indeed, (\ref{ode1}) is equivalent, by the change of variables $v=\beta(z_{M})$,
to
\begin{align*}
\partial_{t}v =\gamma \circ \beta^{-1}(v)\text{,}\quad
v(0) =\beta(M)\text{.}%
\end{align*}
Hence, as $\gamma\circ\beta^{-1}$ is at most of affine growth, $v$ is defined
globally in time. Thus, $T_{0}$ can be chosen arbitrarily large.

Let us now give an estimate for life-span of $u$. Thanks to
Proposition \ref{l infty est prop}, the approximating sequence $u_{\mu}$ can
be chosen such that $u_{\mu}\geq z_{\delta}^{\mu}$ a.e.~in $\Omega
\times (0,T_{0})$. Note that $z_{\delta}^{\mu}$ is positive and uniformly
(with respect to $\mu$) bounded over $[0,T_{0}]$, i.e., $0\leq z_{\delta}%
^{\mu}(t)\leq C$ for all $t\in\lbrack0,T_{0}]$. Moreover, using an equivalent form
$\mu\gamma(\partial_{t}z_{\delta}^{\mu})+\beta^{\prime}(z_{\delta}^{\mu})\partial_{t}z_{\delta}^{\mu}=\gamma (z_{\delta}^{\mu})$ of \eqref{zmu} and using inequality $\beta^{\prime}(z_{\delta}^{\mu}%
)\geq\beta^{\prime}(C)>0$, we can easily obtain uniform bounds on
$\partial_{t}z_{\delta}^{\mu}$ which allow us to pass to the limit as
$\mu\rightarrow0$ and prove that the limit
\[
z_{\delta}(t):=\lim_{\mu\rightarrow0}z_{\delta}^{\mu}(t)
\]
solves $\beta^{\prime}(z_{\delta})\partial_{t}z_{\delta}=\gamma(z_{\delta})$,
$z_{\delta}(0)=\delta$. As a consequence, we have
\[
z_{\delta}(t)=\lim_{\mu\rightarrow0}z_{\delta}^{\mu}(t)\leq\lim_{\mu
\rightarrow0}u_{\mu}(x,t)=u(x,t)\text{ for a.e.~}(x,t)\in\Omega\times
(0,T_{0})\text{.}%
\]
In particular, $T_{0}$ is smaller than the blow-up time for $z_{\delta}$ which
is given by
\[
\hat{T}(\delta)=\int_{\delta}^{+\infty} \dfrac{\beta^{\prime}(r)}{\gamma(r)} \,\mathrm{d}%
r = \int^{+\infty}_{\beta(\delta)} \dfrac{\d b}{\gamma \circ \beta^{-1}(b)}\text{.}%
\]
This completes the proof of Theorem \ref{main thm}. \qed

\section{Weak solvability for non-negative initial data} \label{sec deg data}

In this section, we discuss solvability of (P) for \emph{non-negative} initial data $u^{0}$; more precisely, $u^{0}$ is allowed to vanish on a subset of $\Omega$ with positive measure. Note that, in this case, \eqref{beta growth} is no longer available (see also Remark \ref{R:ex}), and hence, estimates for $\partial_{t}\beta(u)$ (or for $\beta^{\prime}(u)$) do not follow in the same way as before due to the singularity of $\beta^{\prime}$ at $0$. As a consequence, no estimate for $\Delta u $ in a Lebesgue space is derived. Here we shall employ a weaker notion of solutions for the problem $(P)$. To this aim, we assume that
\begin{equation}\label{B}
 \sqrt{\beta'} \in L^1_{loc}([0,+\infty)),
\end{equation}
which will be used to define an auxiliary function $B$ in \eqref{B-def} below (concerning Remark \ref{R:ex}, $\alpha$ is constrained to be less than 2).
Let us recall that
\[
\partial I_{[0,+\infty)}(s)=
\begin{cases}
\{0\} & \text{if } \ s>0\text{,}\\
(-\infty,0] & \text{if }\ s=0\text{.}%
\end{cases}
\]
Then for $\partial_t u, \xi \in L^2(\Omega)$, inclusion $\xi\in\partial_{L^{2}(\Omega)}I(\partial_{t}u)$ is equivalent
to the following
 \begin{align}
\partial_{t}u \geq0 \ \text{ a.e.~in }\Omega \times (0,\infty)\text{,}\nonumber\\%\label{pre weak def 1}\\
\int_{\Omega}\xi\psi \leq0 \quad \text{ for all }\psi\in L^{2}(\Omega)\text{, }%
\psi\geq0 \ \text{ a.e.~in }\Omega\text{,}\label{pre weak def 2}\\
\int_{\Omega}\xi\partial_{t}u \geq0.\label{pre weak def 3}%
 \end{align}
By virtue of equation (\ref{target eq}) along with (\ref{pre weak def 2}), we note that any (regular) solution $u$ to $(P)$ must satisfy
\[
\int_{\Omega}\left(  -\partial_{t}\beta\left(  u\right)  +\LAM\Delta
u+\gamma(u)\right)  \psi\leq0 \quad \mbox{ for all } \psi \in L^2(\Omega) \mbox{ satisfying } \psi \geq 0 \mbox{ a.e.~in } \Omega.
\]
Integrate both sides in time, integrate by parts and choose $\psi\in H^{1}(0,t;H^{1}(\Omega))$ satisfying $\psi(t)=0$ and $\psi\geq0$. We then find that
\begin{equation}
\int_{0}^{t}\int_{\Omega}\beta(u)\partial_{t}\psi+\int_{\Omega} \left( \beta (u^{0})\psi(0)-\LAM\nabla u \cdot \nabla\psi+\gamma(u)\psi \right) \leq  0.\label{weak def}%
\end{equation}
Defining the function $B$ by
\begin{equation}\label{B-def}
B(s)=\int_{0}^{s}\sqrt{\beta^{\prime}(r)} \, \mathrm{d}r \quad \mbox{ for } \ s \geq 0
\end{equation}
and integrating (\ref{pre weak def 3}) in time, we get
\begin{align}
0  & \leq\int_{0}^{t}\int_{\Omega}\left(  -\partial_{t}\beta\left(  u\right)
+\LAM\Delta u+\gamma(u)\right)  \partial_{t}u\nonumber\\
& = -\int_{0}^{t}\int_{\Omega}\left\vert \partial_{t}B(u)\right\vert^{2}-\frac{\LAMuno}{2}\left\Vert \nabla u(t)\right\Vert_2^{2}+\frac
{\LAMuno}{2}\left\Vert \nabla u^{0}\right\Vert_2^{2}+\int_{0}^{t}%
\int_{\Omega}\gamma(u)\partial_{t}u\text{.}\label{weak def 2}%
\end{align}
 We note that, for $u$ smooth enough, \eqref{weak def} and \eqref{weak def 2} imply 
$$
-\partial_{t}\beta\left(  u\right) + \LAM\Delta u+\gamma(u) \in \partial_{L^2(\Omega)} I(\partial_t u).
$$
Moreover, we stress that (\ref{weak def}) and (\ref{weak def 2}) involve neither $\partial_t \beta(u)$ nor $\beta^{\prime}(u)$. In order to take the advantage of this formulation, we introduce a weaker notion of solutions to $(P)$ in the following (cf. \cite{RoRo} for an analogous definition):

\begin{definition}[Weak solutions] A function
$$
u\in L^{\infty}(\Omega\times(0,T))\cap L^{2}%
(0,T;H^{1}(\Omega))\cap H^{1}(0,T;L^{2}(\Omega))
$$ 
is called a \emph{weak solution} to problem $(P)$ on $[0,T]$ if
$B(u)\in H^{1}(0,T;L^{2}(\Omega))$ and $u$ satisfies $u(0)=u^{0}$, $\partial_{t}u\geq0$ a.e.~in
$\Omega\times(0,T)$ and \emph{(\ref{weak def})-(\ref{weak def 2})} for all $t \in (0,T]$.
\end{definition}

The following theorem is concerned with existence of a weak solution to $(P)$ for nonnegative initial data $u^{0} \geq 0$.

\begin{theorem}[Existence of weak solutions for non-negative data]
Let assumptions \emph{\textbf{(A1)}}, \emph{\textbf{(A3)}} and \eqref{B} be satisfied.
Moreover, let $u^{0}\in H^{2}(\Omega)\cap L^{\infty}(\Omega)$  satisfy $\partial_\nu u^0 =0$ on $\partial \Omega$ and  $u^{0}\geq 0$ a.e.~in $\Omega$. Then, there exists $T_0 > 0$ such that $(P)$ admits a weak solution on $[0,T_0]$.
\end{theorem}

\begin{proof}
For every $m\in\mathbb{N}$, let us define an approximating initial datum
$u_{m}^{0}$ by
\[
u_{m}^{0}=u^{0}+1/m\text{. }%
\]
Let $\left(  u_{m},\xi_{m}\right)  $ be a strong solution on $[0,T_0]$ of $(P)$ with the
initial condition $u_{m}(0)=u_{m}^{0}$ for some $T_0 > 0$. Indeed, existence is guaranteed by
Theorem \ref{main thm}. Since the family of initial data $\{u_{m}^{0}\}$ is
uniformly bounded in $L^{\infty}(\Omega)$, local existence time $T_0$ can be taken uniformly in $m$ (see Theorem \ref{main thm}), and moreover, 
\begin{equation}\label{um-ub}
\sup_{t\in\lbrack0,T_0]}\left\Vert u_{m}\left(  t\right)  \right\Vert
_{L^{\infty}\left(  \Omega\right)  }\leq C
\end{equation}
for some $C \geq 0$. Testing equation (\ref{target eq}) with
$\partial_{t}u_{m}$ and recalling $\left(  \xi_{m},\partial_{t}u_{m}\right)
_{L^{2}(\Omega)}=0$ a.e.~in $(0,T_0)$, we get
\begin{equation}
\int_{\Omega}\beta^{\prime}(u_{m})\left\vert \partial_{t}u_{m}\right\vert
^{2}+\frac{\LAMuno}{2}\frac{\mathrm{d}}{\mathrm{d}t}\left\Vert \nabla
u_{m}\right\Vert _{L^{2}(\Omega)}^{2}=\int_{\Omega}\gamma(u_{m})\partial
_{t}u_{m}\text{.}\label{st}%
\end{equation}
Thanks to \eqref{um-ub}, repeating a similar argument to \S \ref{S:conv}, we can derive the
following uniform estimates%
\begin{align*}
\sup_{t\in\lbrack0,T_0]}\left\Vert \xi_{m}(t)\right\Vert _{L^{2}(\Omega)}
& \leq C\text{,}\\
\sup_{t\in\lbrack0,T_0]}\left\Vert \nabla u_{m}(t)\right\Vert _{L^{2}(\Omega)}
& \leq C\text{,}\\
\left\Vert \partial_{t}u_{m}\right\Vert _{L^{2}(0,T_0;L^{2}(\Omega))}  & \leq
C\text{.}%
\end{align*}
By extraction of a not relabeled subsequence, we get the following
convergences as $m\rightarrow+\infty$:%
\begin{alignat}{4}
 u_{m}&\rightarrow u \quad &&  \text{ weakly * in }L^{\infty}(\Omega \times (0,T_0))\text{,}\nonumber\\
 &  \quad &&\text{ weakly * in }L^{\infty}(0,T_0;H^{1}(\Omega))\text{,}\nonumber\\
 &  \quad &&\text{ weakly in }H^{1}(0,T_0;L^{2}(\Omega)),\label{weak conv um}\\
 &  \quad &&\text{ strongly in } C([0,T_0];L^{2}(\Omega)) \text{,}\nonumber\\
 \beta(u_{m})&\rightarrow\beta(u) \quad && \text{ strongly in } C([0,T_0];L^2(\Omega)),\nonumber\\
 \gamma(u_{m})&\rightarrow\gamma(u) \quad && \text{ weakly * in }L^{\infty}(\Omega \times (0,T_0))\text{,}\nonumber\\
 & &&\text{ strongly in } L^2(0,T_0;L^2(\Omega)),\nonumber\\
 \xi_{m} &\rightarrow \xi \quad && \text{ weakly * in }L^{\infty}(0,T_0;L^2(\Omega)).\nonumber%\text{,}\\
%\nabla u_{m} &\rightarrow\nabla u \quad && \text{ weakly * in }L^{\infty}(0,T;L^{2}(\Omega))\text{,}\\
\end{alignat}
Furthermore, we immediately observe that $u_{m}^{0}\rightarrow u^{0}$ and $\beta
 (u_{m}^{0})\rightarrow\beta(u^{0})$ strongly in $L^{2}(\Omega)$.

Finally, fix $t\in(0,T_0]$ and test equation (\ref{target eq}) with some $\psi\in
H^{1}(0,t;H^{1}(\Omega))$ satisfying $\psi(t)=0$, $\psi\geq0$ a.e.~in
$\Omega\times (0,t)$. Since $\xi_{m}\leq0$ a.e.~in $\Omega \times (0,t)$, by integration by parts, one gets
\[
\int_{0}^{t}\int_{\Omega}\beta(u_{m})\partial_{t}\psi+\int_{\Omega}\beta
(u_{m}^{0})\psi(0)-\int_{0}^{t}\int_{\Omega}
\LAM\nabla u_{m} \cdot \nabla\psi+\gamma(u_{m})\psi
=\int_{0}^{t}\int_{\Omega}\xi_{m}\psi\leq0\text{.}%
\]
Using the above convergences and taking a limit, we derive (\ref{weak def}). Note by (\ref{st}) that $\partial
_{t}B(u_{m})=\sqrt{\beta^{\prime}(u_{m})}\partial_{t}u_{m}$ is uniformly
bounded in $L^{2}(0,T_0;L^{2}(\Omega))$. Moreover, since $B$ is continuous and
increasing with $B(0)=0$, we have, by \eqref{st},
\[
\left\Vert B(u_{m})\right\Vert _{H^{1}(0,T_0;L^{2}(\Omega))}\leq C,
\]
whence follows,
\[
B(u_{m})\rightarrow\bar{B} \quad \text{ weakly in } H^{1}(0,T_0;L^{2}(\Omega))
\]
to some limit $\bar{B}\in H^{1}(0,T_0;L^{2}(\Omega))$. Using the continuity
of $B$ and \eqref{st}, we easily prove that $\bar{B}=B(u)$. We further derive as in (\ref{minty}) that
\begin{align*}
&  0\leq\limsup_{m\rightarrow\infty}\int_{0}^{t}(\xi_{m},\partial_{t}%
u_{m})_{L^{2}(\Omega)}\\
&  \leq-\frac{\LAMuno}{2}\liminf_{m\rightarrow\infty}\left\Vert \nabla
u_{m}(t)\right\Vert _{L^{2}(\Omega)}^{2}+\frac{\LAMuno}{2}\left\Vert
\nabla u^{0}\right\Vert _{L^{2}(\Omega)}^{2}\\
& \quad -\liminf_{m\rightarrow\infty}\int_{0}^{t}\left(  \partial_{t}\beta
(u_{m}),\partial_{t}u_{m}\right)  _{L^{2}\left(  \Omega\right)  }%
+\lim_{m\rightarrow\infty}\int_{0}^{t}\left(  \gamma(u_{m}),\partial_{t}%
u_{m}\right)  _{L^{2}\left(  \Omega\right)  }\\
&  \leq-\frac{\LAMuno}{2}\left\Vert \nabla u(t)\right\Vert _{L^{2}%
(\Omega)}^{2}+\frac{\LAMuno}{2}\left\Vert \nabla u^{0}\right\Vert
_{L^{2}(\Omega)}^{2}+\int_{0}^{t}\left(  \gamma(u),\partial_{t}u\right)
_{L^{2}\left(  \Omega\right)  }\\
& \quad -\liminf_{m\rightarrow\infty}\int_{0}^{t}\left\Vert \partial_{t}%
B(u_{m})\right\Vert _{L^{2}(\Omega)}^{2}\\
&  \leq-\frac{\LAMuno}{2}\left\Vert \nabla u(t)\right\Vert _{L^{2}%
(\Omega)}^{2}+\frac{\LAMuno}{2}\left\Vert \nabla u^{0}\right\Vert
_{L^{2}(\Omega)}^{2} +\int_{0}^{t}\left(  \gamma(u),\partial_{t}u\right)  _{L^{2}\left(
\Omega\right)  } - \int_{0}^{t}\left\Vert \partial_{t}B(u)\right\Vert
_{L^{2}(\Omega)}^{2}\text{,}%
\end{align*}
which implies (\ref{weak def 2}). Finally, since $\partial_{t}u_{m}\geq0$ a.e.~in $\Omega \times (0,T_0)$, thanks
to convergence (\ref{weak conv um}), we deduce that $\partial_{t}u\geq0$ a.e.~in $\Omega \times (0,T_0)$.
Thus, $u$ is a weak solution to $(P)$ on $[0,T_0]$.
\end{proof}

\end{document}